\documentclass[12pt]{article}

\usepackage{booktabs} 
\usepackage{amssymb,amsmath,amsthm,sectsty,url,mathrsfs,graphicx}
\usepackage[letterpaper,hmargin=1.0in,vmargin=1.0in]{geometry}
\usepackage[colorlinks,linkcolor=black,citecolor=blue,filecolor=black,urlcolor=blue]{hyperref}
\RequirePackage[OT1]{fontenc}
\RequirePackage[numbers]{natbib}
\usepackage[textsize=tiny]{todonotes}

\usepackage{a4wide}
\parskip \medskipamount

\sectionfont{\large}
\subsectionfont{\normalsize}
\numberwithin{equation}{section}

\sectionfont{\large}
\subsectionfont{\normalsize}
\numberwithin{equation}{section}

\newtheorem{maintheorem}{Theorem}
\newtheorem{theorem}{Theorem}[section]
\newtheorem{lemma}[theorem]{Lemma}
\newtheorem{proposition}[theorem]{Proposition}
\newtheorem{corollary}[theorem]{Corollary}
\newtheorem{fact}[theorem]{Fact}
\newtheorem{definition}[theorem]{Definition}
\newtheorem{remark}[theorem]{Remark}
\newtheorem{conjecture}[theorem]{Conjecture}

\newtheorem{question}[theorem]{Question}

\def\E{\mathop{\mathbb E}}

\newcommand{\ac}{\mathrm{AC}}
\newcommand{\as}{\mathrm{a.s.}}

\newcommand{\iid}{\mathrm{i.i.d.}}
\newcommand{\con}{\mathbf{Con}}

\newcommand{\la}{\lambda}
\newcommand{\lac}{\lambda_{\mathrm{c}}}

\newcommand{\fc}{\mathrm{FC}}
\newcommand{\brw}{\mathrm{LBRW}(\mu_{\la},o)}
\renewcommand{\Pr}{ \mathrm P}

\newcommand{\gl}{\lambda}

\newcommand{\Pois}{\mathrm{Pois}}
\newcommand{\Pl}{\mathrm{P}_{\lambda}}

\newcommand{\ZZ}{\mathbb{Z}_{+} \cup \{\infty\} }

\newcommand{\T}{{\mathbb T}}
\newcommand{\TT}{\mathcal{T}}
\newcommand{\RR}{\mathcal{R}}
\newcommand{\AAA}{\mathcal{A}}
\newcommand{\BB}{\mathcal{B}}

\newcommand{\DD}{\mathcal{D}}
\newcommand{\F}{{\cal F}}

\newcommand{\C}{{\cal C}}

\newcommand{\w}{{\mathbf w}}
\newcommand{\deq}{\stackrel{d}{=}}

\newcommand{\simt}{\stackrel{t}{\sim}}
\newcommand{\siminfty}{\stackrel{\infty}{\sim}}
\newcommand{\meett}{\stackrel{t}{\leftrightarrow}}
\newcommand{\meetinfty}{\stackrel{\infty}{\leftrightarrow}}

\newcommand{\sfrac}[2]{\mbox{\small $\frac{#1}{#2}$}}
\newcommand{\ssfrac}[2]{\mbox{\footnotesize $\frac{#1}{#2}$}}
\newcommand{\half}{\ssfrac{1}{2}}

\newcommand{\N}{\mathbb N}
\newcommand{\R}{\mathbb R}
\newcommand{\Z}{\mathbb Z}

\newcommand{\W}{\mathcal W}

\DeclareMathSymbol{\leqslant}{\mathalpha}{AMSa}{"36} 
\DeclareMathSymbol{\geqslant}{\mathalpha}{AMSa}{"3E} 
\DeclareMathSymbol{\eset}{\mathalpha}{AMSb}{"3F}     
\renewcommand{\le}{\;\leqslant\;}                   
\renewcommand{\ge}{\;\geqslant\;}                   

\let\phi=\varphi
\renewcommand{\epsilon}{\varepsilon}

\newcommand{\Ind}[1]{\mathbf{1}\{#1\}}

\usepackage[normalem]{ulem}

\begin{document}

\title{The social network model on infinite graphs}
\author{Jonathan Hermon
\thanks{
University of Cambridge, Cambridge, UK. E-mail: {\tt jonathan.hermon@statslab.cam.ac.uk}. Financial support by
the EPSRC grant EP/L018896/1.}
\ Ben Morris
\thanks{
Department of Mathematics, UC Davis, USA. E-mail: {\tt morris@math.ucdavis.edu}.}
\
Chuan Qin
\and
Allan Sly 
\thanks{
Department of Mathematics, Princeton University, USA. E-mail: {\tt asly@math.princeton.edu}.}}


\date{}
\maketitle

\begin{abstract}
Given an infinite connected regular graph $G=(V,E)$,  place at each vertex Poisson($\la$) walkers performing independent lazy simple random walks on $G$ simultaneously. When two walkers visit the same vertex at the same time they are declared to be
acquainted. We show that when $G$ is vertex-transitive and amenable, for all $\la>0$ $\as$ any pair of walkers will eventually have a path of acquaintances between them. In contrast, we show that when $G$ is non-amenable (not necessarily transitive) there is always a phase transition at some $\lac(G)>0$. We give general bounds on $\lac(G)$ and study the case that $G$ is the $d$-regular tree in more detail. Finally, we show that in the non-amenable setup, for every $\la$ there exists a finite time $t_{\la}(G)$ such that $\as$ there exists an infinite set of walkers having a path of acquaintances between them by time $t_{\la}(G)$. 
\end{abstract}


\newpage

\tableofcontents

\newpage


\section{Introduction}
We consider the following model for a social network which we call the \emph{social
network model}  (SN as a shorthand). The model was proposed by Itai Benjamini and was first investigated in \cite{BHK} in the context of finite graphs (see \S\ref{s: related} for further details). In this work we study the model on infinite graphs. Let $G=(V,E)$
be an infinite connected $d$-regular graph,
which is the underlying graph of the SN model. 
In our model we have walkers performing independent lazy simple random
walks on $G$, denoted by \emph{LSRW} (see \S\ref{sec: preliminaries} for a definition).
The walkers perform their LSRWs simultaneously (\textit{i.e.}, at each time unit they
all perform one step, which may be a lazy step). The SN model on a graph $G$ with density $\lambda>0$ is defined by setting  $(|\W_v|)_{v \in V}$ to be i.i.d.~$\mathrm{Pois}(\lambda)$ r.v.'s, where $\W_v$ denotes the set of walkers whose initial position is $v$ (and $\mathrm{Pois}(\lambda)$ is the Poisson distribution of mean $\la$). We denote the corresponding probability measure by $\Pr_{\lambda}$.

Let $t \in \Z_+ \cup \{\infty \}$. We say that two
walkers $w,w'$ have \emph{met by time} $t$, which we denote by $w \meett w' $, if
there exists $t_{0} \le t$ such that they have the same position at time
$t_{0}$. After two walkers meet they
continue their walks independently without coalescing. We write $w \meetinfty w'$ (``meeting by time $\infty$"), if there exists some finite $t$, such that $w \meett w'$. 
 ``Meeting
by time $t$" is a symmetric relation and thus induces a unique minimal equivalence
relation that contains it. We call this equivalence relation \emph{having
a path of acquaintances
by time} $t$ and denote it by $\simt$ (note that $w \siminfty w' $ iff there exists some finite $t$ such that $w \simt w' $). More explicitly, two walkers $a$ and $b$ have a path of acquaintances
by time $t$ iff   there exist $k\in \N$, and walkers $a=c_{0},c_{1},\ldots,c_{k},c_{k+1}=b$ such that $c_{i} \meett c_{i+1} $,
 for all $0 \le i \le k$. Note that we are not requiring the sequence of times in which the walkers met to be non-decreasing, which is the main difference between the SN model and some existing models for spread of rumor/infection (\textit{e.g.}, the $A+B \mapsto 2B$ model \cite{kesten2005spread} and the frog model \cite{telcs1999branching,
alves2002shape,popov2001frogs,hoffman2014frog}). Consequently, the SN model typically evolves much faster than such models. 

We are interested
in the coalescence process of the equivalence classes, and in particular in the number of equivalence classes of $\siminfty$ and in the existence of an infinite equivalence class of $\simt$ for some finite $t$.           

Let $\W:=\cup_{u \in V}\W_u $ be the set of all walkers.
Denote by $\mathbf{Con}$ (a shorthand for ``connected") the event that $w \siminfty w' $ for all $w,w' \in \W$ (\textit{i.e.}, $\mathbf{Con}$ is the event that every pair of walkers eventually have a path of acquaintances between them). 
The following question was proposed to us by Itai Benjamini \cite{benjaminipc}.
\begin{question}
Let $\mathbb{T}_d$ be the infinite $d$-regular tree. Does  $\Pr_{\lambda}[\mathbf{Con}]=1$
 for all $\lambda>0$? \end{question}
We give a negative answer to this question (Theorem \ref{thm: darytree}).
This raises the problem of identifying for which graphs $\con$ occurs $\Pr_{\lambda}$-$\as$ for all $\lambda>0$.
\begin{definition}[Critical density]
\label{def: lac}
Let $G=(V,E)$ be an infinite connected regular graph. The \emph{critical density} for the SN model on $G$ is defined to be 
$$\lac(G):=\inf\{\lambda: \exists \, p>0 \text{ such that } \inf_{u,v \in V} \Pr_{\la}[u \siminfty v \mid \W_u \neq \eset, \W_v \neq \eset ] \ge p \},$$
where for a pair of vertices $u,v$ and $t \in \Z_+ \cup \{\infty \}$ we write $ u \simt v $ (respectively, $u \meett v $) if there exist some $w \in \W_v$ and $w' \in \W_u$ such that $w \simt w'$ (respectively, $w \meett w'$).
\end{definition}
The following phase transition occurs around the critical density: 
\begin{proposition}
\label{prop: criticalintro}
Let $G$ be an infinite connected regular graph. Then
\begin{equation}
\label{eq: lacintro}
\Pr_{\lambda}[\con]=\begin{cases}0 & \text{if  }0 < \lambda<
\lac(G), \\
1 & \text{if  }\lambda >
\lac(G).
\end{cases}
\end{equation}
\end{proposition}
A graph $G=(V,E)$ is called \emph{vertex transitive} if the action of its automorphisms group, $\mathrm{Aut}(G)$, on its vertices is transitive (\textit{i.e.}, $\{\phi(v):\phi \in \mathrm{Aut}(G) \}=V $ for all $v$). The spectral radius of a random walk on $G=(V,E)$ with transition kernel $P$ is $\rho:=\limsup_{n}(P^n(v,u))^{1/n} $ (the limit is independent of $u,v \in V$). A graph   $G$ is called \emph{amenable} if $\rho=1$ for LSRW on $G$ (otherwise, it is called \emph{non-amenable}). We review some consequences of amenability/non-amenability in \S\ref{sec: gap} and \S\ref{sec: amenable}. 

There are numerous characterizations of amenability. Most characterizations describe a certain dichotomy between amenable graphs and non-amenable graphs. In particular, several probabilistic models exhibit very different behaviors in the amenable setup and the non-amenable setup. However, proving a sharp dichotomy may be an extremely challenging task for some models. For instance, it is a major open problem in percolation theory to establish that for vertex transitive graphs, the existence of a non-uniqueness regime for Bernoulli percolation is equivalent to non-amenability. For
further details see \cite[Chapter 7]{lyonsbook}. For a different recent characterization of non-amenability via percolation see \cite{Tom}.
The following theorem asserts that for transitive graphs, amenability can be characterized by the SN model (note that there is no transitivity assumption in the non-amenable setup).
\begin{maintheorem}
\label{thm: amenabilityintro}
For every infinite connected vertex transitive amenable graph of finite degree, $\lac=0$. Conversely, for every infinite non-amenable connected regular graph $\lac>0$.
\end{maintheorem}
\begin{remark}
A similar dichotomy is believed to hold for the frog model (in the context of recurrence), however the only family of non-amenable graphs for which a phase transition is known to exist in the frog model is the infinite $d$-regular tree for all $d \ge 3$ \cite{hoffman}. The frog model in the amenable setup is studied in  \cite{teixeira2013random}. \end{remark}
\begin{remark}
Using our analysis of the non-amenable setup it is not hard to verify that by attaching the root of an infinite binary tree to the origin of $\Z^d$ we obtain a non-transitive amenable graph with $\lac>0$. Thus the transitivity assumption is necessary in Theorem \ref{thm: amenabilityintro}.
\end{remark}

A question which arises naturally is what can be said about $\lac$ in the non-amenable setup.   We give general lower and upper bounds on $\lac(G)$ (Theorems \ref{thm: easylower} and \ref{thm: mainupper}, respectively) in terms of the spectral radius  $\rho$  of the walk and the degree $d$ of the underlying graph. It turns out that the holding probability (which obviously affects $\rho$) can drastically change $\lac$, which is somewhat counter-intuitive at first sight. As an illustrating example we consider the infinite $d$-regular tree.

\begin{maintheorem}
\label{thm: darytree}
Let $\T_d=(V,E)$ be the infinite $d$-regular tree for some $d \ge 3$. There exist absolute constants $c,C>0$ such that when the holding probability of the walks is  taken to be $1/(d+1)$ we   have that\begin{equation}
\label{eq: lacTd}
c \sqrt{d} \le \lac (\T_d) \le C \sqrt{d}.
\end{equation}
\end{maintheorem} 
In contrast, Theorem \ref{thm: mainupper} asserts that when the holding probability is taken to be $1/2$, there exists an absolute constant $C$ such that for all $d \ge 3$ and all infinite connected $d$-regular graphs $G$ we have that $\lac (G) \le C \log d $.   In \S\ref{s:refined} we state and provide a sketch of proof of Theorem \ref{thm: hardlower}, which refines Theorem \ref{thm: easylower} and asserts the following. There exists an absolute constant  $c>0$ such that  for every connected,  infinite, regular graph $G$, when the holding probability is $1/2$ we have that $\lac(G) \ge c \log (1/\rho) $, where $\rho
$ is the spectral-radius of simple random walk on $G$ (rather than of lazy simple random walk with holding probability 1/2). 

Combining these two results  we obtain as a corollary that $c \log d \le \lac (\T_d) \le C \log d  $, when the holding probability is $1/2$. In fact, the same bounds hold for all infinite connected Ramanujan graph, which are by definition $d$-regular graphs with $\rho=\frac{2\sqrt{d-1}}{d}$. (For SRW on infinite connected $d$-regular graphs one always has that  $\rho \ge \frac{2\sqrt{d-1}}{d} $  and for $\T_d$ this is an equality -- see \textit{e.g.}, \cite[Theorem 6.10]{lyonsbook}.)

We strongly believe that up to the value of the absolute constants, the same bounds hold for the continuous-time analog of the SN model, as the ones holding in discrete-time when the holding probability is $1/2$.

\subsection{Infinite friend clusters in finite time}

We now turn our attention to the problem of determining the existence of an infinite equivalence class of $\simt$ for some finite $t$. 

Let  $t \in \Z_{+} \cup \{\infty\}$. For each walker $w$ we call the set walkers in the same equivalence class of $ \simt  $ as $w$, the \emph{friend cluster} of $w$ \emph{at time} $t$ and denote it by $\mathrm{FC}_{t}(w)$. When $t=\infty$ we call this set the \emph{friend cluster}
of $w$ and denote it by $\mathrm{FC}(w):=\mathrm{FC}_{\infty}(w)$. More generally, when $t=\infty$ we often omit it from our terminology and notation. Recall that for $u,v \in V$ and $t \in \Z_{+}\cup \{\infty \} $ we denote $u \meett v$ and $u\simt v$ iff there exist $w \in \W_u$ and $w' \in \W_v$ so that  $w \meett w'$ and $w \simt w'$, respectively. Let \[\Xi:=\{u \in V: |\W_u|>0 \} \] be the set of initially occupied vertices. It will be convenient to define the friend cluster of a vertex $u$ at time $t$, which by abuse of notation we denote by $\mathrm{FC}_{t}(u)$, which is defined as follows. If $u \in \Xi$ then we define $\fc_t(u)$ to be the friend cluster of the walkers in $\W_u$ at time $t$, \textit{i.e.}, $\mathrm{FC}_{t}(u):=\mathrm{FC}_{t}(w)$ for some (and hence every) $w \in \W_u $. Otherwise, we set $\mathrm{FC}_{t}(u) $ to be the empty-set. Note that
\begin{equation}
\label{eq: deflac}
\lac(G)=\inf\{\lambda: \exists \, p>0 \text{ such that } \forall \, u,v \in V,\quad \Pr_{\la}[\fc(u)=\fc(v) \mid u,v \in \Xi ] \ge p \}.
\end{equation}
Minor adjustments to the analysis of the frog model on $(\Z/n\Z)^d $ from \cite{frogs} show that  when the underlying graph is $\Z^d$ with $d>1$, for every $\la>0$ there is indeed an infinite friend cluster in finite time $\as$ 

\begin{conjecture}[Benjamini \cite{benjaminipc}]
\label{conj: infinitecluster}
Let $G$ be an infinite connected graph of bounded degree. Assume that for some $0<p<1$ Bernoulli bond percolation on $G$ with survival probability $p$ has an infinite connected component with probability 1. Then for all $\lambda>0$, there exists $t_{\lambda}(G)>0$ such that
$$ \forall \, t > t_{\lambda}(G) , \quad \Pr_{\lambda}\left[\max_{w} \mathrm{|FC}_{t}(w)|=\infty  \right]=1.$$
\end{conjecture}
The following theorem provides a partial answer.
\begin{maintheorem}
\label{thm: infcluster}
Let $G=(V,E)$ be a regular connected infinite non-amenable graph. Denote the spectral radius of LSRW with some arbitrary holding probability $0 \le p<1$ by $\rho$. Let $\mathrm{IC}(t) $ be the event that $\max_{v \in V}|\fc_t(v)|=\infty$.  Then there exists an absolute constant $C>0$ (independent also of $G$) such that for all $\la \in (0,1]$ and $t \ge \lceil \frac{C}{\la(1-\rho)} \rceil=:t_{C,\la} $
$$\Pl[\mathrm{IC}(t  )]=1. $$ 
\end{maintheorem}
\begin{remark}
Theorem 6 in \cite{BHK} asserts that (for $\la=1$) if $G$ is a $d$-regular expander of size $n$, then there exists some constants $t,c_1$ (depending only on the spectral gap of the walk on $G$ and on $d$) such that after $t$ steps $\max_v |\fc (v)| \ge n/6 $ with probability at least $1-e^{-c_1n}$. However, the techniques from the finite setup do not carry over  to the infinite setup of Theorem \ref{thm: infcluster}. 
\end{remark}
\subsection{Related work}
\label{s: related}
The SN model,  proposed by Itai Benjamini, was first investigated in the context of finite graphs and $\gl=1$  in \cite{BHK}, where it was shown that   there exist constants $c,C>0$ such that for  every finite connected graph $G=(V,E)$  of average degree $d$, $$  \Pr[c \log |V| \le \inf \{t: \mathrm{FC}_t(u)=\mathrm{FC}_{t}(v) \text{ for all }u,v \in \Xi \} \le Cd^6 \log^{3} |V|
 ] \ge 1-|V|^{-1}. $$
 That is, (when $\la=1$ and the holding probability is taken to be $1/2$) the first time at which all walkers have a path of acquaintances between them is with high probability poly-logarithmic in the number of vertices, provided that the average degree is at most poly-logarithmic. For $d$-regular graphs the term $d^6$ is improved to $d$. Further improvements are given under  appropriate heat-kernel decay assumption or under a certain isoperimetric assumption.  
\subsection{Organization of the paper and discussion of our techniques}
In \S\ref{sec: preliminaries} we present some preliminaries about Poisson thinning, percolation and random walks on non-amenable graphs. In \S\ref{sec:prop1.3} we prove Proposition \ref{prop: criticalintro}. 

In \S\ref{sec: amenable} we prove the assertion of Theorem \ref{thm: amenabilityintro} in the transitive amenable setup (namely, that $\lac=0$). The main tools used in \S\ref{sec: amenable} are borrowed from the study of percolation. Namely, we consider the graph with vertex set $V$ in which all vertices in $V \setminus \Xi $ are isolated and each pair of vertices $u,v \in \Xi$ are connected  if $\fc(u)=\fc(v)$ (\emph{i.e.}, if eventually there is a path of acquaintances between the walkers whose initial location is $u$ and the ones whose initial location is $v$). We wish to show that for every $\la>0$  $\as$ all $u \in \Xi$ lie in the same connected component (this is the same as saying $\Pl[\con]=1$) 

 We show that this percolation process stochastically dominates an auxiliary translation-invariant percolation process possessing insertion tolerance (see \S\ref{s:percolationdefinitions} for the relevant definitions), in which for each $u \in \Xi$ the  connected component of $u$ is infinite. Using standard machinery from the theory of percolation on transitive amenable graphs (see Theorem \ref{thm: longrangeuniqueness}) we deduce that the auxiliary percolation process $\as$ has a unique infinite cluster. The vertex  set of the unique infinite cluster must be $\Xi$, as if some $u \in \Xi$ does not lie in the unique infinite cluster, then there would be more than one infinite cluster (as the cluster of $u$ is infinite, as is the cluster of every $v \in \Xi$). The aforementioned stochastic domination implies that $\Pl[\con]=1 $.

 In \S\ref{s:upperthm1} we bound $\lac$ from above in the non-amenable setup. The idea of the argument is to argue that if $\la$ is sufficiently large, then any two friend clusters have a drift towards each other. Clearly, if $\la$ is large enough (in terms of the degree) this is true in the first step. The idea is to exploit Poisson thinning, and to somehow use just a fraction of the walkers, in a manner that guarantees that at each step we have a sufficient amount of ``unused" walkers to maintain a drift. The key fact used in the analysis is the exponential decay (w.r.t.\ time) of the transition probabilities of the random walk.  

In \S\ref{s:tree} we consider the $d$-ary tree and prove Theorem \ref{thm: darytree}. Here we use a certain comparison between the SN model with parameter $\lambda$ and a Bernoulli bond percolation, with parameter proportional to $(\lambda/d)^2 $, on a certain copy of $\T_{\lceil d/2 \rceil} $ inside $\T_d $.  This percolation is supercritical if $\la > C \sqrt{d} $, which by the nature of the comparison we establish, in turn implies the supercriticality of the SN model.

 In \S\ref{s:thm3} we prove Theorem \ref{thm: infcluster}. Here we use a variant of an exploration process of Benjamini, Nachmias and Peres \cite{benjamini} which they used to prove locality of the critical percolation probability for non-amenable graphs of large girth. Their analysis establishes some connection between percolation and random walks, and hence it is perhaps not surprising that a variant of it is useful also in our setup. 

 In \S\ref{s:lowernonamen} we conclude the proof of Theorem \ref{thm: amenabilityintro} by proving a general lower bound on $\lac$ in the non-amenable setup. Here we explore the friend cluster in a way which we then dominate by a  branching random walk with mean offspring distribution $1+2\la$. Such a branching random walk is known to be transient provided that $1+2\la \le 1/\rho$ \cite{criticalBRW}, where $\rho$ is the spectral-radius of the corresponding walk (see \S\ref{sec: gap} and \S\ref{s:lowernonamen} for definitions). Transience of the branching random walk implies that $\as$ there are some vertices which are never visited by walkers in the friend cluster of the origin. 

If we only considered paths of acquaintances which are monotone in time (as in the aforementioned $A +B \mapsto 2B$ and frog models -- see the discussion at page 3), then as we now explain it would have been relatively easy to dominate the friend cluster via a branching random walk with offspring distribution whose law is the same as that of $1+X$, where $X \sim \Pois(\lambda) $. For this consider the exploration process in which at each time unit $t$ we recruit to the exploration process (the yet unrecruited) walkers that met at time $t$ one of the walkers already recruited to the exploration process before time $t$. Using Poisson thinning it is not hard to argue that each recruited walker contributes at each stage at most $\Pois(\lambda)$ new walkers. 

In  \S\ref{s:lowernonamen} we describe a variant of this exploration process, which actually captures the evolution of $\fc(o)$. Exploring directly the evolution of $\fc_t(o) $ as time increases is counter-productive, as it grows to rapidly and by Theorem \ref{thm: infcluster} it becomes infinite in finite time. Instead we shall explore the evolution of  $\fc(o)$ in a slowed down fashion.    At each stage we reveal two steps of each previously recruited walker, one corresponds to moving forward in time, as above, and the other corresponds to a step backwards in time. Namely, if a walker is recruited to the exploration process at stage $t$, due to an acquaintance made at time $s \le t$, then at stage $t+i$ (for $i >0$) we reveal its location at time $s+i $ (forward step) and if $i \le s$ also its location at time $s-i$ (backwards step).  

 Crucially, using reversibility (and the fact that the transition kernel $P$ of the walk performed by each particle is symmetric, \textit{i.e.}, it satisfies $P(x,y)=P(y,x)$ for all $x,y$), if $(v_0,v_1,\ldots) $ is a random walk, then $(v_t,v_{t-1},\ldots,v_1,v_0)$ is also a random walk. That is, the  backwards evolution of each walker still has the law of a random walk. Thus we may think of each recruited walker as two distinct particles, one corresponding to the forward trajectory, and one to the backwards trajectory (from the time at which the walker was recruited until time 0). This accounts for the multiplicative term $2$ in $1+2\lambda$ above. Namely, we dominate the exploration process via a branching random walk with offspring distribution $1+2X$, where $X \sim \Pois(\la) $. 

Finally, in \S\ref{s:concluding} we give a refinement of the lower bound on $\lac$ from \S\ref{s:lowernonamen} which is specialized to the case where the holding probability is large. It is used to determine the asymptotic behavior of $\lac(\T_d)$ as $d \to \infty$ when the holding probability is $1/2$, as described after the statement of Theorem \ref{thm: darytree}.   
  
\section{Preliminaries and additional notation}
\label{sec: preliminaries}
LSRW is defined as follows.  If a
walker's current position is $v$, then the walker either stays in its current position w.p.\ $1/2$, which we refer to as the \emph{holding probability}, or moves to one of the neighbors of $v$ w.p.~$\frac{1}{2d}$. We shall also consider the case of holding probability $1/(d+1)$ in which $1/(2d)$ and $1/2$ above are both replaced by $1/(d+1)$.
\subsection{Reversibility, Poisson thinning, stationarity of the occupation measure and independence of the number of walkers performing different walks.}
\label{s:revthin}
Let $G=(V,E)$ be a regular graph. Then the transition kernel $P$ of LSRW on $G$ is symmetric (\emph{i.e.}, $P(x,y)=P(y,x)$ for all $x,y \in V $) and so $P^t$ is also symmetric for all $t \in \N$. In other words, $P$ is \emph{reversible} w.r.t.\ the counting measure on $V$. We now establish a certain independence property for walks in $G$, which in particular implies stationarity of the occupation measure for the SN model. 

A \emph{walk} of length $k$ in $G$ is a sequence of $k+1$ vertices $(v_0,v_1,\ldots,v_{k})$ such that for all $0 \le i <k$ either $v_i=v_{i+1}$ or $\{v_{i} , v_{i+1}\} \in E$. Let $\Gamma_k$ be the collection of all walks of length $k$ in $G$. Throughout,
we denote the set of walkers whose initial position is $v$ by $\W_v:=\{w_1^{v},\ldots,w_{N_{v}}^v \}$. We denote by $\w_{i}^v(t)$ the position of the walker $w_{i}^v$ at time $t$. We say that a walker $w_{i}^v$ \emph{performed a walk} $\gamma \in \Gamma_k$ if $(\w_{i}^v(0),\ldots,\w_{i}^v(k)) =\gamma$.  For a walk $\gamma=(\gamma(0),\ldots,\gamma(k)) \in \Gamma_{k}$ for some $k \ge 1$, we denote $p(\gamma):=\prod_{i=0}^{k-1}P(\gamma_i,\gamma_{i+1}) $. This is precisely the probability that some given walker $w \in \W_{\gamma(0)}$ performed the walk $\gamma$. 

Let $\gamma_{\mathrm{rev}} $ be the reversal of $\gamma \in \Gamma_k$. That is $\gamma_{\mathrm{rev}}(i)=\gamma(k-i) $ for all $0 \le i \le k $. Then by reversibility $p(\gamma)=p(\gamma_{\mathrm{rev}}) $.  We denote
the number of walkers whose position at time $t$ is $v$ by $Y_{v}(t)$. 
By reversibility, for all $v \in V $ and $t>0$ we have  $\mathbb{E}_{\lambda}[Y_v(t)]=  \sum_{u
\in V} \mathbb{E}_{\lambda}[Y_{v}(0)]P^{t}(u,v)= \la \sum_{u
\in V} P^{t}(v,u)= \la $. Thus by Poisson thinning:
\begin{fact}
\label{fact: thinning}
Let $G=(V,E)$ be a regular graph. Denote the number of walkers who performed a walk $\gamma$ (in the above sense) by $X_{\gamma}$. For every $\lambda>0$, under $\Pr_{\la}$ we have that $X_{\gamma} \sim \Pois(\lambda p(\gamma)) $, for all $t>0$ and $\gamma \in \Gamma_t$.  Moreover, $(X_{\gamma})_{\gamma \in \Gamma_t} $ are independent for each fixed $t>0$. Consequently,  $(Y_v(t))_{v \in V} $ are $\iid$ $\Pois(\la)$ random variables for each fixed $t>0$.
\end{fact}
\subsection{Further notation, monotonicity and the regeneration Lemma}
\label{sec: construction}

Let $t \in \ZZ $. The \emph{acquaintances
graph} at time $t$, denoted by $\mathrm{AC}_{t}(G)=(V,E_{t})$,
is a random graph in which two distinct vertices
$u,v \in V $ are connected by an edge iff $u \meett v $. We denote $\ac(G):=\ac_{\infty}(G)$.
 We denote the
connected component of $v$ in $\mathrm{AC}_{t}(G)$ by $C_{t}(v)$. Note that $\fc_t(v)= \bigcup_{u \in C_{t}(v) }\W_u $, where as before $\W_u$ is the set of walkers which initially occupy vertex $u$. When clear from context, we omit $G$ from the notation. When we want to emphasize the density of the walkers we write $\ac_{t}^{\la}(G)$. We denote the collection of walkers which occupy vertex $v$ (respectively, the set $A \subseteq V $) at time $t$ by $\W_v(t) $ (respectively, $\W_A(t) $) and set $\W_A:= \W_A(0)= \bigcup_{a \in A }\W_a $ (this is the set of walkers whose initial position is in $A$). 

\begin{proposition}
\label{prop: couplingforalllambda}
Let $G=(V,E)$ be a regular graph. There exists a probability space on which the SN model on $G$ is defined for all  $\lambda>0$ simultaneously, such that  deterministically, for all $t \in
\Z_{+} \cup \{\infty \} $ we have that $\ac_{t}^{\lambda_1}$ is a subgraph
of  $\ac_{t}^{\lambda_2}$ for all $\lambda_1 \le \lambda_2$. 
\end{proposition}
The construction is fairly straightforward and is very similar to the one in \cite{frogs}. We present it in the Appendix \ref{A:construction} for the sake of completeness. 

\begin{lemma}(Regeneration Lemma)
\label{lem: regen}
Let $G=(V,E)$ be an infinite $d$-regular graph.  Let $Y_{v,B}(t)$ be the number of walkers belonging to  $\W_B$ which are at vertex $v$ at time $t$. Then for every finite set $A \subset V $  and each fixed $t$, $(Y_{v,A^{\complement}}(t))_{v \in V}$ are independent Poisson r.v.'s, where $A^{\complement}:=V \setminus A $ is the complement of $A$. Moreover, $\lim_{t \to \infty} \inf_{v} \mathbb{E}_{\la}[Y_{v,A^{\complement}}(t)]=\la $.
\end{lemma}
\begin{lemma}
\label{lem: infcollide}
Let $G=(V,E)$ be an infinite, connected, regular graph. Let $\mathbf{w}:=(\mathbf{w}_0,\mathbf{w}_1,\ldots)$  $\in V^{\Z_+}$. For $t \in \N \cup \{ \infty \} $ let $N_t(\mathbf{w}) $ be the number of  walkers not belonging to $\W_{\mathbf{w}_0}$ which for some $i \le t $ visited $\mathbf{w}_i$ at time $i$. Then  $N_t(\mathbf{w}) $ has a Poisson distribution  for all $t$ and $\mathbf{w}$ whose mean (under $\Pl$) is at least $c \la \sqrt{t} $ for some constant $c$. In particular, $N_{\infty}(\mathbf{w}) $ is infinite~$\as$      
 \end{lemma}

The proofs of the last two lemmas involve straightforward applications of Fact \ref{fact: thinning} combined with the general bound $\sup_{x,y \in V }P^{t}(x,y) \le \frac{C}{\sqrt{t+1}}$   and are thus deferred to Appendix \ref{a:regen}.
 
\subsection{Insertion tolerance, translation invariance, ergodicity.}
\label{s:percolationdefinitions}
We now show how the SN model on a graph  $G$ with a countable vertex set $V$  can be viewed as a \emph{long-range bond percolation
process} on $G$. This will allow us to use existing machinery from percolation theory in our study of the SN model.  

Let $S:= \{\{v,u \}:v \neq u,\, v,u \in V \} $. The \emph{standard form} of a probability space of a long-range bond percolation
process on $G$ is $(\{0,1\}^S,\Pr,\F_{\mathrm{cylinder}})$, where $\F_{\mathrm{cylinder}}$ is the the cylinder $\sigma$-algebra, the minimal  $\sigma$-algebra w.r.t.~which $\{x \in \{0,1\}^S :x(s)=1 \} $ is measurable for all $s \in S$. Each $x \in \{0,1\}^S$ can be viewed as a graph $\mathrm{graph}(x)=(V,E(x))$, where $s \in E(x)$ iff $x(s)=1$, in which case we say $s$ is \emph{open} in the configuration $x$. If $x(s)=0$ we say that $s$ is \emph{closed} in the configuration $x$. For $\mathcal{B} \subseteq \{0,1\}^S $ we write $\mathrm{graph}(\mathcal{B}):=\{\mathrm{graph}(b):b \in \mathcal{B} \}$.

 Let $(\Omega, \Pr,\F)$ be a probability space in which there exist zero-one valued random variables $(Z_{s})_{s \in S}$ ($S$ as above). This probability space gives rise to a (long-range bond) \emph{percolation process} on $G$ as follows. For every $\omega \in \Omega$ we construct a graph $\mathrm{graph}(\omega)=(V,E(\omega))$ by setting $s \in E(\omega)$ iff $Z_{s}(\omega)=1$. Note that  $\omega \mapsto \mathrm{graph}(\omega) $ need not be bijective.

Several  definitions which we soon give take a simple form when the percolation process is given in the standard form. These definitions extend to the general case as follows.
There is a canonical correspondence between $(\Omega, \Pr,\F)$ and a probability space having the standard form. For every $\omega \in \Omega $, we define $\psi (\omega) \in \{0,1\}^{S}$ by setting $\psi (\omega)(s)=Z_s(\omega)
$. For every $\mathcal{B} \in \F$ set $\psi(\mathcal{B}):=\{\psi(b):b \in \mathcal{B} \} \subseteq \{0,1 \}^S $. Conversely, for every $x \in \{0,1\}^S$ we set $\psi^{-1}(x):=\{\omega \in \Omega:\psi (\omega)=x  \} $ and for every $\mathcal{B} \subseteq \{0,1\}^S$ we set $\psi^{-1}(\mathcal{B} ):=\bigcup_{x \in \mathcal{B} }\psi^{-1}(x)$. By abuse of notation, we identify the restriction of $\Pr$ to the $\sigma$-algebra generated by $(Z_{s})_{s\in S}$ with the space $(\{0,1\}^S,\Pr_{\mathrm{cylinder}},\F_{\mathrm{cylinder}})$, where for every $\mathcal{B} \in \F_{\mathrm{cylinder}}  $, $\Pr_{{\mathrm{cylinder}}}(\mathcal{B} ):=\Pr(\psi^{-1}(\mathcal{B} )) $. That is, we identify  $x \in \{0,1\}^S$ and $\mathcal{B}  \in \F_{\mathrm{cylinder}}
 $ with $\psi^{-1}(x)$ and $\psi^{-1}(\mathcal{B} )$, respectively, and by abuse of notation write $\Pr(\mathcal{B} ) $ for $\Pr(\psi^{-1}(\mathcal{B} )) $. In particular, we say that $\Pr$ satisfies one of the properties defined below if $\Pr_{\mathrm{cylinder}} $ satisfies this property.

For every $x \in \{0,1\}^S$ and $s \in S $, we define $x_{s}^{+} \in \{0,1\}^S $ by setting $$x_{s}^{+}(s'):=\begin{cases}1 & s'=s, \\ 
x(s') & \text{otherwise}
\end{cases}.$$ That is, $x_{s}^{+}$ is obtained from $x$ by flipping the value at $s$ to 1 if necessary, while keeping the  the configuration unchanged elsewhere.  For every $s \in S$ and $\mathcal{B} \subseteq \{0,1\}^S$ we define $$\mathcal{B}_{s}^{+}:=\{b_{s}^{+}:b \in \mathcal{B} \}.$$ Note that $\mathrm{graph}(\mathcal{B} _{s}^{+})=\{(V,E(b) \cup \{s\} ):b \in \mathcal{B} \}$ (where as before $ \mathrm{graph}(b)=(V,E(b) )$; In other words, if we identify $\mathcal{B}_s^+ $ and $\mathcal{B}$ with collections of graphs, then the former is obtained from the latter by adding to each graph in $\BB$ the edge $s$, if it did not already appear in it).  We say that $\Pr$ is  \emph{insertion tolerant} (also known as having \emph{positive
finite energy}) if for all $\mathcal{B}  \in \F_{\mathrm{cylinder}} $ such that $\Pr(\mathcal{B} )>0$ also  $\Pr(\mathcal{B} _{e}^{+})>0$, for all $e \in E$.

Every $\phi \in \mathrm{Aut}(G) $ acts on $\{0,1\}^S$ ($\phi:\{0,1\}^S \to \{0,1\}^S $) via $\phi(x)(s)=x(\phi(s)) $. Clearly, $\mathrm{graph}(\phi(x))$ is isomorphic to $\mathrm{graph}(x)$.
We say that an event $\mathcal{A}  \in \F_{\mathrm{cylinder}}$ is  \emph{translation invariant} 
if for all
$\phi \in \mathrm{Aut}(G)$ we have that $\mathcal{A}=\phi(\mathcal{A})$, where $\phi(\mathcal{A}):=\{\phi(a):a \in \mathcal{A} \}$. We denote the $\sigma$-algebra of all translation invariant events by $\mathcal{I} $. We say that $\Pr$ is \emph{translation invariant} if for all $\mathcal{A}\in \F_{\mathrm{cylinder}}$ we have that $\Pr(\mathcal{A})=\Pr(\phi(\mathcal{A})) $ for all $\phi \in \mathrm{Aut}(G)$. When the percolation process is defined via Bernoulli random variables $(Z_s)_{s \in S}$, this is equivalent to the requirement that for all $\phi \in \mathrm{Aut}(G)$ we have that $(Z_{s})_{s \in S} \deq (Z_{\phi (s)})_{s
\in S}$, where $\deq$ denotes equality in distribution.
We say that $\Pr$ is \emph{ergodic} if $\Pr(\mathcal{A}) \in \{0,1\}$ for
all  $\mathcal{A} \in \mathcal{I}
$.

\begin{proposition}
\label{prop: ergodicity}
Let $G=(V,E)$ be an infinite connected vertex-transitive graph. Then for all $\lambda > 0$ we have that the law of $\ac_{t}^{\lambda}(G)$ is translation invariant and ergodic for all $\lambda>0$ and $t \in \Z_{+} \cup \{\infty \}$.
\end{proposition}
When $G$ is a Cayley graph, it is straightforward to see that $\ac_{t}^{\lambda}(G) $ is a factor of i.i.d.'s and hence is indeed translation invariant and ergodic. When $G$ is only assumed to be transitive one can still present  $\ac_{t}^{\lambda}(G) $ as a factor of i.i.d.'s, but this requires some care. We defer the proof of Proposition \ref{prop: ergodicity} to Appendix \ref{a:ergodicity}.   
\subsection{Couplings and stochastic domination}
Let $G=(V,E)$ be a graph. As before, let $S:= \{\{v,u \}:v \neq u,\, v,u \in V \} $.  Equip  $\{0,1\}^S$ with the partial order
$\leqslant$, where $x \leqslant y$ iff $x(s)\le
y(s)$ for all $s \in S$. We say that $\mathcal{A} \in \F_{\mathrm{cylinder}} $ is increasing if $x \in
\mathcal{A} $ and $x \leqslant y$ imply that also $y \in
\mathcal{A}$. For any two probability
measures on $(\{0,1\}^S ,\F_{\mathrm{cylinder}})$, $\mu$ and $\nu$, we say that $\mu$ stochastically
dominates $\nu$ if $\mu(A) \ge \nu(A)$ for every increasing event $\mathcal{A}  \in \F_{\mathrm{cylinder}}
$.

 Let $(X_s)_{s \in S}$ and $(Y_s)_{s
\in S}$
be Bernoulli random variables defined on the same probability space $(\Omega,\Pr,\F)$. Let the marginal distributions of $(X_s)_{s \in S}$ and $(Y_s)_{s \in S}$ under $\Pr$ be $\mu$ and $\nu$, respectively. Such a construction is called a {\emph{coupling} of $\mu$ and $\nu$. It is
well-known and straightforward to show that if there exists
such
a coupling in which for all $s \in S$, $X_s \ge Y_s$
$\Pr$-$\mathrm{a.s.}$, then $\mu$ stochastically
dominates $\nu$.
Thus by Proposition \ref{prop: couplingforalllambda}:
\begin{proposition}
\label{prop: monotonicity}
For every underlying graph $G$ for the SN model we have that
for all $t \in
\Z_{+} \cup \{\infty \} $, the law of $\ac_{t}^{\lambda_2}$ stochastically dominates the law of  $\ac_{t}^{\lambda_1}$ for all $\lambda_1 \le \lambda_2$.  
\end{proposition}
\subsection{Non-amenability and the spectral radius}
\label{sec: gap}
Let $G=(V,E)$ be a connected infinite regular graph. Let $\pi $ be the counting measure on $V$. The space of $L_2$ functions is given by  $\ell^2(V,\pi) :=\{f \in  \R^V: \|f \|_2<\infty \}$, where $\|f \|_2^2:=\langle f,f \rangle $ and $\langle f,g \rangle:=\sum_v f(v)g(v)$). Let $K$ be a symmetric (\textit{i.e.}, $K(x,y)=K(y,x)$ for all $x,y \in V $) transition kernel of a Markov chain $(X_t)_{t=0 }^{\infty}$ on $V$. We identify it with an operator by setting $(Kf)(x):=\sum_y K(x,y)f(y)=\E_x[f(X_1)]$. Its \emph{operator norm} is given by
\begin{equation}
\label{e:opnorm}
\|K\|:=\sup \{\sfrac{\|Kf\|_2}{\|f\|_2} :f \in \ell^2(V,\pi), f \neq 0  \}= \sup \{\sfrac{\langle K f,f \rangle}{\|f\|_2^2} :f \in \ell^2(V,\pi), f \neq 0  \} \end{equation} 
 (\textit{e.g.}, \cite[Ex.\ 6.7]{lyonsbook}). Let $x,y \in V$ be arbitrary vertices. The \emph{spectral
radius} of $K$ is 
\begin{equation}
\label{eq: spectralradiusdefK}
\rho(K):=\limsup_{n\rightarrow \infty} [K^{n}(x,y)]^{1/n}.
\end{equation}
 It is standard that (see \textit{e.g.}, \cite[p.\ 182-183]{lyonsbook}): \begin{itemize}
\item[(1)] The limit is independent of the choice of $x,y$. 
\item[(2)] $\rho(K)=\|K\| $.
\item[(3)] $K^{n}(x,y) \le [\rho(K)]^{n} $ for all $x,y$ and $n \ge 0$ (use $K^{n}(x,y)=\langle K^{n} 1_x ,1_y \rangle $ and (2)).  
\end{itemize}

Let $0 \le p<1$. Let $P_p$ be the transition kernel
of LSRW on $G$ with holding probability $p$ (\textit{i.e.}, $P_p=pI+(1-p)P_0$, where $P_0$ corresponds to simple random walk on $G$).  Let $x,y \in V$ be arbitrary vertices. We denote the \emph{spectral
radius} of $P_p$ by 
\begin{equation}
\label{eq: spectralradiusdef}
\rho_{p}:=\rho(P_{p})=\limsup_{n\rightarrow \infty} (P_{p}^{n}(x,y))^{1/n}.
\end{equation}
We denote the spectral radius of the SRW by $\rho(G):=\rho_0$. By (3) above
\begin{equation}
\label{eq: spectralxy}
P_{p}^{n}(x,y) \le \rho_{p}^{ n}, \text{ for all }x,y \in V \text{
and }n \ge 0 .
\end{equation}
Thus having $\rho_{p}<1$ is equivalent to having uniform exponential decay of
the transition probabilities w.r.t.~$P_p$. By (2) above, \eqref{e:opnorm} and the fact that \[\langle P_{p} f,f \rangle =p\langle f,f \rangle+(1-p)\langle P f,f \rangle \] we have that 
\begin{equation}
\label{e:rhop}
\rho_p=p+(1-p)\rho(G), 
\end{equation}
and so $\rho(G)<1 $ iff $\rho_p<1$ for all $p \in [0,1)$.
%
\section{Proof of Proposition \ref{prop: criticalintro}}
\label{sec:prop1.3}
\begin{proof}
We first note that if $\lac >0$ and $0<\la < \lac$, then there exists a sequence $(u_n,v_n)_{n \in \N } \subset V \times V $ such that $\Pr_{\la}[\fc(u_{n})=\fc(v_{n}) \text{ and } u_{n},v_{n} \in \Xi ] \le 2^{-n}$, for all $n$. By (both parts of) the Borel-Cantelli Lemma, $\Pr_{\la}$-$\as$ there exists some $n$ such that $  \fc(u_{n}) \neq \fc(v_{n})$ and 
 $ u_{n},v_{n} \in \Xi$. Indeed, on the one hand, $\as$ there are only finitely many $n$'s such that $\fc(u_{n})=\fc(v_{n}) \text{ and } u_{n},v_{n} \in \Xi$, while on the other hand, $\as$ there are infinitely many $n$'s such that $ u_{n},v_{n} \in \Xi$. Thus $\as$ there exists some $n$ such that  $ u_{n},v_{n} \in \Xi$ and $\fc(u_{n})=\fc(v_{n})$. Thus $\Pr_{\la}[\con]=0$, as desired. 

Conversely, fix some $\la > \lac \ge 0$. We shall show that $\Pr_{\la}[\con]=1$.
 By definition of $\lac$ (and the monotonicity of the model w.r.t.~$\la$)\footnote{Actually, we are using here also the fact that the Poisson($\la$) distribution conditioned on being positive is stochastically increasing in $\la$. To see this, consider the number of points in $[0,1]$ for a rate $\la$ Poisson process. Observe that conditioned on having at least 1 point, the location of the first point is stochastically decreasing in $\la$. Given that the first point is at $x$ the number of additional points has a $\Pois(\la(1-x))$ distribution (which is stochastically decreasing in $x$ and increasing in $\la$). We leave the remaining details to the reader.} there exists some $p>0$ such that  $\inf_{u,v}\Pr_{\la'
}[\fc(u)=\fc(v) \mid u,v \in \Xi] \ge p$  for all $\la' \ge \frac{\lambda+\lac}{2} $. Fix some $u,v \in V$. Let us condition on  $u,v \in \Xi $. Let $B_r$ be the ball of radius $r$ around $u$. Let $D_t$ be the event that there exist some $k \in \N$ and some $u_1=u,u_2,\ldots,u_{k+1}=v$ all belonging to $B_t$ such that $u_i \meett u_{i+1} $ for all $1 \le i \le k $. Since $D_{t} \nearrow  \{ u \siminfty v \} $ as $t \to \infty$ (recall that a path of acquaintances has a finite length) there exists a finite time $t_1$ and some finite set $A_1 \subset V$ (both may depend on $u,v$), such that w.p.~at least $p/2$, there exists a path of acquaintances between the walkers from $\W_u$ and $\W_v$ by time $t_1$, which only uses walkers from $\W_{A_1}:= \cup_{w \in A_1}\W_w $.  We think of this as the ``first trial" to connect the walkers in $\W_u$ to those in $\W_v$. 

Using the regeneration Lemma we show that after each failed trial, there will be another trial whose success probability is at least $p/2$, regardless of the information exposed in all previous trials. All trials involve some finite set of walkers and a finite amount of time (both may depend on the  information exposed in previous trials).

Denote by $Y_{a,1}(t)$ the number of walkers not from $\W_{A_{1}}$ which are at vertex $a$ at time $t$. By Lemma \ref{lem: regen}, there exists some $s_1$ so that $(Y_{a,1}(s))_{a \in V}$ stochastically dominate $\iid$ $\Pois(\la_2)$ random variables for all $s \ge s_1$, where $\la_2:=\lac +\frac{3(\lambda-\lac)}{4}$. We may assume that $s_1=t_1$ by increasing one of them if necessary.
Pick some $w_u \in \W_u$ and $w_v \in \W_v$ and let $(\mathbf{w}_u(t))_{t=0}^{\infty},(\mathbf{w}_v(t))_{t=0}^{\infty}$ be the LSRWs they perform, respectively. Let $\W_a(t)$ be the collection of walkers which are at vertex $a$ at time $t$.

Repeating the same reasoning as before (with $\la_2=\lac +\frac{3(\lambda-\lac)}{4}$ in the role of $\la$) yields that there must exist some $t_2>t_1$ and some finite set $A_2 \subset V$ (both may depend on $(\mathbf{w}_v(t_1),\mathbf{w}_u(t_1)) $) such that given the walks performed by the walkers in $\mathcal{A}_1:= \W_{A_{1}}$ by time $t_1$ (and that the first trial failed) we have that (i)-(ii) below hold: 
\begin{itemize}
\item[(i)] The conditional probability  that $w_u$ and $w_v$  have a path of acquaintances by time $t_2$ which uses only walkers from $\mathcal{(A}_2\setminus \mathcal{A}_1)\cup \{w_v,w_u \} $ where $\mathcal{A}_2:=\cup_{a \in A_2} \W_{a}(t_1)  $,  and all the acquaintances along this path were made between time $t_1$ and $t_2$ (ignoring possible earlier acquaintances if such occurred), is  at least $p/2$. 
\item[(ii)]  $(Y_{a,2}(t_2))_{a\in V}$ stochastically dominate $\iid$ $\Pois(\la_3)$ r.v.'s, where $\la_3:=\lac +\frac{5(\lambda-\lac)}{8}
 $ and $Y_{a,2}(t_2)$ is the number of walkers, not from $\mathcal{A}_1 \cup \mathcal{A}_2$, which are at $a$ at time $t_2$. 
\end{itemize}
It is clear how to continue. Namely, by induction on $i$ one can argue that there exist $t_{i+1}$ and finite sets $A_1,\ldots,A_{i+1} \subset V $ and $\mathcal{A}_j:=\cup_{a \in A_j} \W_{a}(t_{j-1})  $ for $j \in [i+1] $ (where $t_0:=0$ and both $t_{i+1} $ and $A_{i+1} $ may depend on $(\mathbf{w}_v(t_{i+1}),\mathbf{w}_u(t_{i+1})) $)) such that $t_{i+1}>t_i$ and given the walks performed by the walkers in $\cup_{j \in [i]}\mathcal{A}_{j} $ by time $t_{i} $ we have that (i)-(ii) below hold:
\begin{itemize}
\item[(i)] The conditional probability  that $w_u$ and $w_v$  have a path of acquaintances by time $t_{i+1}$ which uses only walkers in \[\mathcal{(A}_{i+1} \setminus \cup_{j=1}^{i}  \mathcal{A}_{j}) \cup \{w_v,w_u \} ,\]  and all the acquaintances along this path were made between time $t_{i}$ and $t_{i+1}$ (ignoring possible earlier acquaintances if such occurred) is  at least $p/2$. 
\item[(ii)]  $(Y_{a,i+1}(t_{i+1}))_{a\in V}$ stochastically dominate $\iid$ $\Pois(\la_{i+2})$ random variables, where $\la_{i+2}:=\lac +\frac{(2^{i+1}+1)(\lambda-\lac)}{2^{i+2}}
 $ and $Y_{a,i+1}(t_{i+1})$ is the number of walkers, not from $\cup_{j \in [i+1]}\mathcal{A}_{j}  $, which are at $a$ at time $t_{i+1}$. 
\end{itemize}
As each  trial has success probability at least $p/2$, regardless of the result of the previous rounds, $\as$ one of the trials will be successful, where here success means that the event from (i) occurs. 
\end{proof}
\section{The amenable case}
\label{sec: amenable}
We shall utilize the following theorem, taken from \cite{gandolfi1992uniqueness}, in our analysis of the amenable case. We note that in  \cite{gandolfi1992uniqueness} only the graphs $\Z^d$  for $d \in \N$ (or some half spaces) were considered. However their analysis can easily be extended to all amenable vertex-transitive graphs.
\begin{theorem}
\label{thm: longrangeuniqueness}
Let $G=(V,E)$ be an infinite connected vertex-transitive amenable graph. Let $(\Omega,\Pr) $ be a translation invariant long range bond percolation process on $G$ possessing insertion tolerance. Then $\Pr [\text{there exists at most one infinite connected component}]=1$.
\end{theorem}
For  $\mathcal{B} \in \F_{\mathrm{cylinder}}$ and  $e \in E$ let $ \widehat{\mathcal{B}_e}:=\{(V,F): e \in F, F \supseteq F' \text{ for some }(V,F') \in \mathcal{ \mathrm{graph}(\mathcal{B})} \} $ be the collection of all graphs obtained by adding to each graph in $\mathrm{graph}(\mathcal{B})$ some collection of edges containing $e$. 

Note that for all $\mathcal{B} \in \F_{\mathrm{cylinder}}$ and  $e=\{u,v\} \in E$, by planting additional walkers at $u$ and $v$ (this is done in the proof below) we see  that     \[ \Pr_{\la}[\ac \in \mathrm{graph}(\mathcal{B})]>0 \Longrightarrow \Pr_{\la}[\ac \in \widehat{\mathcal{B}_e}]>0. \]

 The problem is that planting additional walkers at $u$ and $v$ might add more than just the edge $\{u,v\}$ to $\ac$. Thus this idea cannot be used to establish insertion tolerance. In order to utilize Theorem \ref{thm: longrangeuniqueness}, we construct an auxiliary model, stochastically dominated by the SN model, to which this idea applies. In order to ensure we can add to the obtained graph with positive probability an edge $e$ and \emph{only} that edge, in the auxiliary model the planted walkers can only make acquaintances at time 1. 
\begin{theorem}
\label{thm: amenablecase}
Let $G=(V,E)$ be an infinite connected vertex-transitive amenable graph. Then $\lac=0$. 
\end{theorem}
\begin{proof}
Let $\la >0$. We partition the particles into two independent sets, $\W^1,\W^2$ of density $\la/2$ each. We may consider the evolution of the model only w.r.t.\ $\W^1$  (as if $\W^2$ did not exist). Denote the obtained acquaintances graph w.r.t.\ $\W^1$ for time $ \infty $ by $H:=(V,E_{1})$. Denote the degree of $G$ by $d$. We now partition $\W^2 $ into $d$ sets of density $\la/(2d) $ as follows. For $v \in V$ let $N(v):=\{u \in V:\{u,v\} \in E \}$ be the set of its neighbors. Let $\W^i_v $ be the particles in $\W^i$ (where $i \in \{1,2\}$) which initially occupy $v$. We partition it into $d$ sets: $\mathcal{W}(v,u)$ for $u \in N(v)$. Let $E_{2} \subseteq E $ be the collection of edges $\{u,v\} \in E$ such that there is some particle $w \in \mathcal{W}(v,u) $ and some particle $w' \in \mathcal{W}(u,v) $ which met at time 1 (note that this is always possible as we take the holding probability to be positive). 

Let $H_{1}:=(V,E_{1} \cup E_{2})$. Note that by Poisson thinning the events $\{e \in E_{2} \}$ are independent for different $e \in E$ and thus 
  $H_{1}$ is insertion tolerant. The proof of translation invariance of the SN model, with minor adaptations can easily be extended to show that the law of $H_{1}$ is translation invariant.

We may switch the roles of $\W^1 $ and $\W^2$ in the above construction and now partition each $\W^1_v$ further into $d$ sets  $\mathcal{\widehat W}(v,u)$ for $u \in N(v)$ to get: $\widehat H:=(V, \widehat E_{1}) $ the acquaintances graph for time $\infty$ defined only w.r.t.\ $\W^2$ and  $\widehat E_{2} \subseteq E $ the collection of $\{u,v\} \in E$ such that there is some particle $w \in \mathcal{\widehat W}(v,u) $ and some particle $w' \in \mathcal{\widehat W}(u,v) $ which met at time 1. By symmetry also $H_2:=(V,\widehat E_{1} \cup \widehat E_{2}) $ is insertion tolerant and translation invariant.

Clearly, $\widetilde H:=(V, E_{1}  \cup E_2 \cup \widehat E_{1} \cup \widehat E_{2})$ is a subgraph of the (usual) acquaintances graph for time $\infty$ (when the walkers are not partitioned into different sets). Thus it suffices to argue that $\as $ it has a unique infinite connected component containing all $u \in \Xi=\{v:\W_v \neq \eset \}$.

It follows from Theorem \ref{thm: longrangeuniqueness} that both $H_1$ and $H_2$ $\as$ have at most one infinite connected component. Now if $\W_v \neq \eset  $, then $\W_v^i \neq \eset $ for some $i \in \{1,2\}$. It is not hard to verify that for all positive $\la'$, the SN model with particle  density $\la'$  satisfies that every $u \in \Xi$ lies in an infinite connected component of the acquaintances graph for time $\infty$,  as every walker meets infinitely many other walkers by time $\infty$ (this follows from Lemma \ref{lem: infcollide}). By uniqueness it follows that every $v$ such that $\W_v^i \neq \eset$ lies in the same infinite connected component of $H_i$. As $\as$ there is some $v$ such that both $\W_v^1 \neq \eset$ and $\W_v^2 \neq \eset$ it follows that  $\widetilde H$ has a unique infinite connected component containing all $u \in \Xi$.         
\end{proof}

\section{An upper bound on the critical density in the non-amenable setup.}
\label{s:upperthm1}
\begin{theorem}
\label{thm: mainupper}
Let $G=(V,E)$ be a $d$-regular connected infinite non-amenable graph. Denote the spectral radius of LSRW with holding probability $1/(d+1)$ (respectively, $1/2$) by $\rho$ (respectively, $\rho_{1/2}$). If the holding probability of the walks is  $1/(d+1)$ (respectively, $1/2$) then $\lac \le  (d+1+\frac{2}{1-\rho})\log 8 $ (respectively, $\lac \le \frac{20 \log d}{1-\rho_{1/2}}$).
\end{theorem}
We first explain the main idea behind the proof of Theorem \ref{thm: mainupper} in simple words, in a slightly simpler setup. We concentrate here on the case that the holding probability is $\frac{1}{d+1}$. Let $u,v \in V$. We want to bound the conditional probability, given that $u \in \Xi $ (\textit{i.e.}, that $u$ is initially occupied), that the friend cluster  of some walker $w \in \W_u$ eventually contains some walker which visited $v$. (Note that this need not imply that $u \siminfty v $. Thus in the proof of Theorem \ref{thm: mainupper} we will have to work with two ``paths", rather than one. Namely, we will construct also a path starting from $v$ in such a way that the two paths will collide.) 

Note that the number of particles in $\W_u \setminus \{w \} $ does not have a $\Pois(\la) $ distribution. To deal with this, in the proof of Theorem \ref{thm: mainupper} we shall use the regeneration lemma. But for the sake of the current discussion, let us assume that the walker $w$ was planted at $u$ at time 0, so that $\W_u \setminus \{w \} \sim \Pois(\la) $. Pick some $ \hat u_1 \sim u $ which is closer to $v$ than $u$ is. The number of walkers from $\W_u \setminus \{w \} $ which crossed from $u=u_{0}$ to $\hat u_1$ has a $\Pois(\la/(d+1))$ distribution.

 Fix some $\alpha \le \la/(d+1)$ to be determined shortly. By Poisson thinning we can look at time one at a subset  $\W(1) $  of them whose size has a  $\Pois(\alpha)$ distribution (namely, by including in it each walker which crossed from $u_0$ to $\hat u_1$ at time 1 w.p.\ $\alpha/[\la/(d+1)] $ independently). If it is not empty, we set $u_1:=\hat u_1 $, otherwise, we set $u_1 $ to be the location of $w$ at time 1. 

Assume by induction that we have defined the vertices $u_0,u_1,\ldots,u_i $ and $\hat u_1,\ldots,\hat u_i $ as well as disjoint sets of walkers $\W(1),\ldots,\W(i)$, such that 
\begin{itemize}
\item For all $j \in [i]$ the size of $\W(j)$ has a  $\Pois(\alpha)$ distribution  (given the information exposed up to the time $\W(j)$ was defined; i.e., given  $\W(1),\ldots,\W(j-1)$ as well as  $u_0,u_1,\ldots,u_{j-1} $ and $\hat u_1,\ldots,\hat u_{j-1} $).
\item   For all $j \in [i]$ the set $\W(j)$  is a subset of the set of walkers which was at $u_{j-1} $ at time $j-1$ and then moved to $\hat u_{j}$ at time $j$, where $\hat u_j $ is some neighbor of $u_{j-1}$ which is closer to $v$ than $u_{j-1}$ is. 
\item If $|\W(j)|>0$ we set $u_j=\hat u_j $. Otherwise, we set $u_j $ to be a vertex closest to $v$ which is occupied at time $j$ by some walker from $\cup_{m=0 }^{j-1} \W(m) $ (where $\W(0):=\{w\}$).
\end{itemize}

Observe that if $\alpha$ is sufficiently large, then the sequence $(u_i:i\in \Z_+) $ has a positive drift towards $v$. In order for this construction to work, it is necessary that the distribution of the number of walkers which are at $u_i $ at time $i$, which do not belong to $\cup_{m=0 }^{i} \W(m) $, will stochastically dominate the $\Pois(\alpha(d+1)) $ distribution. In fact, it is not hard to prove by induction that for all $a_1,\ldots,a_{i},\hat a_1,\ldots,\hat a_{i-1} \in V $,  conditioned on  $u_0=a_{0},u_1=a_{1},\ldots,u_i=a_{i} $ and $\hat u_1=\hat a_{1},\ldots,\hat u_{i-1}=\hat a_{i-1} $,  the aforementioned law is a Poisson with parameter $\la - \alpha \sum_{j=1}^{i}p_j $, where $p_j=P^{i-j}(a_j,a_i)$ is the probability of a given walker from $\W(j)$ to be at $u_i$ at time $i$. Clearly, $p_j \le \sup_{x,y}P^{i-j}(x,y) \le \rho^{i-j} $. We get that $\la - \alpha \sum_{j=1}^{i}p_j \le \la - \alpha/(1-\rho) $, and thus the construction is indeed possible, provided that $\la$ is sufficiently large. Crucially, after conditioning on  $u_0=a_{0},u_1=a_{1},\ldots,u_i=a_{i} $ and $\hat u_1=\hat a_{1},\ldots,\hat u_{i-1}=\hat a_{i-1} $  as above, using the induction hypothesis, the induction step requires only a standard use of Poisson thinning.   
  
\emph{Proof of Theorem \ref{thm: mainupper}.}
First consider the case that the holding probability is $(d+1)^{-1}$. Fix some $u,v \in \Xi$ and $\la >  (d+1+\frac{2}{1-\rho})\log 8  $.  We shall construct two random paths (more precisely, two sequences of vertices) $\gamma,\gamma' $  such that the walkers which are at $\gamma_t$ (respectively, $\gamma_t'$) at time $\ell+ t$ (for some $\ell$ to be determined below) are in $\fc(v)$ (respectively, $\fc(u)$). 

Denote the natural filtration of $(\gamma_t,\gamma'_{t})_{t \ge 0}$ by $\F_t $. We will show that there exists some $c>0$ such that for all $t  $  on the event $d(\gamma_{t},\gamma'_{t})>0$ we have      $\mathbb{E}_{\la}[d(\gamma_{t+1},\gamma'_{t+1})-d(\gamma_{t},\gamma'_{t})  \mid \F_t] \le -c$, (where $d(\cdot,\cdot)$ is the graph distance, \textit{i.e.}, the paths have a bias towards each other). This clearly implies that given that $u,v \in \Xi$, we have that $\fc(u)=\fc(v)$ $\Pl$-$\as$

Fix some $\alpha > \log 8$ such that $\la> (d+1+\frac{2}{1-\rho})\alpha $. At time 0 expose some $w_v \in \W_v=:\mathcal{A}_0 ,w_{u} \in \W_u =:\mathcal{B}_0 $ and their locations at time $\ell$ and set $\gamma_0 $ and $\gamma_0'$ to be these locations, resp., where $\ell$ is sufficiently large so that the distribution of the number of walkers, not belonging to $\mathcal{A}_0 \cup \mathcal{B}_0 $ at the different vertices of $G$ at all times $t \ge \ell $, stochastically dominates that of $\iid$ $\Pois(\la')$ random variables, where $\la':=\alpha (d+1+\frac{2}{1-\rho})$. In other words, \begin{equation}
\label{e:kappata}
\inf_{t \ge \ell,a \in V} \kappa_t(a) \ge \la', \quad \text{where}\quad \kappa_t(a):=\la(1-P^t(u,a)-P^t(v,a))  .\end{equation} 

Recall that for $v \in V$ and $t \in \Z_+$ we define $\W_v(t)$ as the set of walkers occupying $v$ at time $t$. For an oriented edge (possibly a loop)  $e=(e^{-},e^+)$  let \[\mathcal{W}_{e}(t):=\W_{e^-}(t) \cap \W_{e^+}(t+1)\]  be the collection of all walkers whose positions at times $t$ and $t+1$ are $e^-$ and $e^+$, respectively. Clearly, it suffices to describe the construction of  $\gamma,\gamma' $ only until the first $k$ for which $\gamma_k=\gamma_k'$.
We define $\gamma,\gamma'$ inductively as follows.   
Assume that $(\gamma_i,\gamma_i')_{i=0}^{k-1}$ and some collection of oriented edges   $e_1,f_1,\ldots,e_{k-1},f_{k-1}$, have already been defined and that for all $1 \le i < k$ in the $i$-th step of the construction we first define $e_i $, then (as described in (3) below) expose a certain set of walkers $\mathcal{A}_i \subseteq \mathcal{W}_{e_{i}}(\ell+i-1) $ and define $\gamma_i$ (as described below in (2)), after which we define $f_i$, expose a set of walkers $\mathcal{B}_i \subseteq \mathcal{W}_{f_{i}}(\ell+i-1)$ and finally define $\gamma_i' $, such that the following hold (the construction is described only in (2) and (3), while (4)-(5) are included as part of the induction hypothesis only for the purpose of facilitating the induction step):
\begin{itemize}
\item[(1)]  $\gamma_i \neq \gamma_i'$ for all $i<k$ (otherwise, the construction is concluded before stage $k$).
\item[(2)] For all $1 \le i<k$, the edge $e_i=(e_i^{-},e_i^+) $ is some oriented edge in $G$  of the form $e_i=(\gamma_{i-1},v_i)$  satisfying that $d(v_{i},\gamma'_{i-1})=d(\gamma_{i-1},\gamma'_{i-1})-1$ (\textit{i.e.}, $v_i$ is some neighbor of $\gamma_{i-1}$ which is closer to $\gamma'_{i-1}$ than $\gamma_{i-1}$ is). The sets 
$\mathcal{A}_0,\mathcal{B}_0,\ldots,\mathcal{A}_{i-1},\mathcal{B}_{i-1} $ have already been defined, as described in (3) below.   The set $\mathcal{A}_i$ is then defined inductively in a manner described in (3) below so that given $|\mathcal{A}_0|,|\mathcal{B}_0|,|\ldots|,|\mathcal{A}_{i-1}|,|\mathcal{B}_{i-1}| $, $e_1,f_1,\ldots,e_{i-1},f_{i-1} $ and $(\gamma_j,\gamma_{j} ')_{j=0}^{i-1} $ 
\begin{equation}
 \label{e:RiAi}
 \mathcal{A}_i \subseteq \mathcal{R}_i:=\mathcal{W}_{e_{i}}(\ell+i-1) \setminus \bigcup_{j=0}^{i-1}\mathcal{(A}_j \cup \mathcal{B}_j) \quad \text{and} \quad |\mathcal{A}_i | \sim \Pois(\alpha) . \end{equation} If $|\mathcal{A}_{i}| \ge 1 $,  we set $\gamma_i=v_{i}=e_i^+ $. Otherwise, we define $\gamma_i$ to be some vertex occupied at time $\ell +i $ by some walker in $\cup_{j=0}^{i-1}\mathcal{A}_j $ of minimal distance from $\gamma_{i-1}' $.

 Similarly, after defining $\gamma_i$, we set $f_i=(f_i^-,f_i^+)$ to be of the form $f_i=(\gamma_{i-1}',u_i)$ satisfying that $d(\gamma_{i},u_{i})=d(\gamma_{i},\gamma'_{i-1})-1  $ if $\gamma_i \neq \gamma_{i-1}'$; otherwise, we set $u_i:=\gamma_i$.  As before,   we then define the set $\BB_i$  inductively in a manner described in (3) below so that given $|\mathcal{A}_0|,|\mathcal{B}_0|,|\ldots|,|\mathcal{A}_{i-1}|,|\mathcal{B}_{i-1}| ,|\mathcal{A}_{i}|$, $e_1,f_1,\ldots,e_{i-1},f_{i-1},e_i, \gamma_i $ and $(\gamma_j,\gamma_{j} ')_{j=0}^{i-1} $     
\begin{equation}
 \label{e:RiAi2} \mathcal{B}_i \subseteq \mathcal{R}_i':= \W_{f_i}(\ell +i-1) \setminus  \bigcup_{j=0}^{i-1} \mathcal{(A}_{j} \cup \mathcal{B}_{j})  \quad \text{and} \quad |\mathcal{B}_i | \sim \Pois(\alpha) .
\end{equation} 
If
$|\mathcal{B}_{i}| \ge 1 $, we set $\gamma_i'=u_{i}=f_i^+ $. Otherwise, we define  $\gamma_i'$ to be some vertex occupied at time $\ell +i $ by some walker in $\cup_{j=0}^{i-1}\mathcal{B}_j $ of minimal distance from $\gamma_{i} $. 
\item[(3)] The sets 
$\mathcal{A}_1,\mathcal{B}_1,\ldots,\mathcal{A}_{k-1},\mathcal{B}_{k-1} $ are all disjoint and their sizes are $\iid$ $\Pois(\alpha)$. Denote
\begin{equation}
\label{eq: betaiy}
\begin{split}
&\beta_i(y):=\sum_{j=1}^{i-1} \left( \mathbb{E}_{\la}[|\mathcal{A}_{j}|]P^{i-j-1}(e_{j}^+,y)+\mathbb{E}_{\la}[|\mathcal{B}_{j}|]P^{i-j-1}(f_{j}^+,y) \right) \\ & =\alpha \sum_{j=0}^{i-2} \left( P^j(e_{i-j-1}^+,y)+P^j(f_{i-j-1}^+,y) \right) \le 2\alpha \sum_{j \ge 0}\rho^j =2 \alpha/(1-\rho).  
\end{split}
\end{equation}
Let  $\mathcal{R}_i$ and $\mathcal{R}'_i$ be as in \eqref{e:RiAi}-\eqref{e:RiAi2}. Let $\kappa_t(\cdot)$ and $\beta_i(\cdot)$ be as in \eqref{e:kappata} and \eqref{eq: betaiy}.  Then, for all $i<k$ given $e_1,f_1,\ldots,e_{i-1},f_{i-1} $, we have that  $|\mathcal{R}_{i-1}|$ and $|\mathcal{R}'_{i-1}| $ are independent Poisson r.v.'s, 
\begin{equation}
\label{eq: RiRi'}
\begin{split}
& \mathbb{E}[|\mathcal{R}_i| \mid e_1,f_1,\ldots,e_{i-1},f_{i-1},e_{i}]=P(e_{i}^-,e_{i}^+)[\kappa_{i+\ell-1}(e_i^-)- \beta_i(e_i^-)] \\ & \ge \sfrac{1}{d+1}(\la' - 2\alpha/(1-\rho) ) \ge \alpha, \\ & \mathbb{E}[|\mathcal{R}_i'| \mid e_1,f_1,\ldots,e_{i},f_{i}]=P(f_{i}^-,f_{i}^+)[\kappa_{i+\ell-1}(f_i^-)- \beta_i(f_i^-)] \ge \alpha.
\end{split}
\end{equation}
 For all $i<k$, the set $\mathcal{A}_i$  (respectively, $\mathcal{B}_i$) is a random subset of $\mathcal{R}_i$ (respectively, $\mathcal{R}'_i$) obtained from it by including in $\mathcal{A}_i$ (respectively, $\mathcal{B}_i$) every element of $\mathcal{R}_i $ (respectively, $\mathcal{R}_i'$) independently w.p.~$p_i:=\alpha/\mathbb{E}[|\mathcal{R}_i| \mid e_1,f_1,\ldots,e_{i-1},f_{i-1},e_{i}  ]$ (respectively, $p_i':=\alpha/\mathbb{E}[|\mathcal{R}_i'| \mid e_1,f_1,\ldots,e_{i},f_{i} ]$). Note that by \eqref{eq: RiRi'} $p_i,p_i' \le 1$.  
\item[(4)]
For $i < k$ and every walk $\w:= (\w_0,\ldots,\w_{\ell +i})$ with $\w_0 \notin \{u,v\} $, given $e_1,f_1,\ldots,$ $e_{i-1},f_{i-1} $,   the number  $Q_{\w}$  of walkers not  belonging to $\bigcup_{j=0}^{i-1}(\mathcal{A}_j \cup \mathcal{B}_j)$ which performed the walk $\w$ has a Poisson distribution.\footnote{ The exact expression for the mean shall not be used in what comes. It  is given by $\la p(\w)\prod_{j \in I_{\w}}(1-p_j)\prod_{j' \in J_{\w}}(1-p_{j'}')$, where $p_j$ and $p_j'$ are as in (3) and $p(\w)$ is as in \S\ref{s:revthin}, and where  \[I_{\w}:=\{ 1 \le j < i:(\w_{\ell+j-1},\w_{\ell+j}) =e_j \}\,  \text{ and }\] \[  J_{\w}:=\{ 1 \le j < i:(\w_{\ell+j-1},\w_{\ell+j}) =f_j \}.\] .} Moreover, for each fixed $i<k$, given $e_1,f_1,\ldots,e_{i-1},f_{i-1} $, the $Q_{\w}$'s (where $\w$ is as above, of length $\ell+i$).

\item[(5)] Consequently, for all $i \le k$ the number $U_y^{i}$  of walkers not  belonging to $\bigcup_{j=0}^{i-1}\mathcal{(A}_j \cup \mathcal{B}_j)$ which are at vertex $y$ at time $\ell+i-1$, has a Poisson distribution (by (4)) with mean  $\kappa_{i+\ell-1}(y)- \beta_i(y)$. Finally, for each fixed $i \le k$ we have that $(U_y^i)_{y \in V}$ are mutually independent.
\end{itemize}
In order to define $e_k,\mathcal{A}_k, \gamma_k,f_k,\mathcal{B}_k, \gamma_k'$ (in this order) we apply steps (2) and (3) with $k$ in the role of $i$. It is not hard to see that by the induction hypotheses (3)-(5) together with Poisson thinning and \eqref{eq: RiRi'} (with $i=k$),  this extends the construction by one step so that (2)-(5) remain valid for $k+1$ in the role of $k$. We leave the details to the reader.

 Note that  $d(\gamma_{k},\gamma'_{k-1}) - d(\gamma_{k-1},\gamma'_{k-1}) \le 1  $ and that also $d(\gamma_{k},\gamma'_{k}) - d(\gamma_{k},\gamma'_{k-1}) \le 1$. By step (2) the first increment equals $-1$ w.p.~at least $ \mathbb{P}(\Pois(\alpha)>0)=1-e^{-\alpha}>7/8$  and the same holds for the second increment, unless $\gamma_k=\gamma_{k-1}' $, in which case the second increment equals 0 w.p.~at least $1-e^{-\alpha}>7/8$. Thus $\gamma$ and $\gamma'$ are indeed biased towards each other as desired. 
\medskip

We now consider the case of holding probability $1/2$. We explain the necessary adaptations leaving some of the details to the reader. Set $\la=\frac{20 \log d}{1-\rho_{1/2}} $. As before let $\gamma_0 $ and $\gamma_0'$ be the positions of $w_v$ and $w_u$ at time $\ell$, respectively, where $\ell$ is so that the distribution of the number of walkers, other than $w_u$ and $w_v$, at the different vertices of $G$ at time $\ell$ stochastically dominates that of $\iid$ $\Pois(\la')$ for some $\la'> \frac{20 \log d}{1-\rho_{1/2}} -1$.   

 Assume that for some collection of oriented edges   $e_1,f_1,\ldots,e_{k-1},f_{k-1}$ the sequence $(\gamma_i,\gamma_i')_{i=0}^{k-1}$ has been defined and that in the $i$-th step of the construction we exposed sets of walkers \[\mathcal{A}_i \subseteq \mathcal{W}_{e_{i}}(\ell+i-1),\] \[\mathcal{C}_i \subseteq \mathcal{W}_{(\gamma_{i-1},\gamma_{i-1})}(\ell+i-1) \mathcal{=W}_{\gamma_{i-1}}(\ell+i-1) \cap \mathcal{W}_{\gamma_{i-1}}(\ell+i) ,\]  \[\mathcal{B}_i \subseteq \mathcal{W}_{f_{i}}(\ell+i-1) \quad \text{and} \]  \[\mathcal{D}_i \subseteq \mathcal{W}_{(\gamma'_{i-1},\gamma'_{i-1})}(\ell+i-1)  = \mathcal{W}_{\gamma'_{i-1}}(\ell+i-1) \cap \mathcal{W}_{\gamma'_{i-1}}(\ell+i) ,\] so that 
\begin{itemize}
\item[(i)]  $\mathcal{A}_1,\mathcal{B}_1,\mathcal{C}_1,\mathcal{D}_1, \ldots,\mathcal{A}_{k-1},\mathcal{B}_{k-1},\mathcal{C}_{k-1},\mathcal{D}_{k-1} $ are all disjoint; 
\item[(ii)]  $|\mathcal{A}_1|,|\mathcal{B}_1|,\ldots,|\mathcal{A}_{k-1}|,|\mathcal{B}_{k-1}| $ are $\iid$ $\Pois(2d^{-1}\log d )$;  \item[(iii)]  $|\mathcal{C}_1|,|\mathcal{D}_1|,\ldots,|\mathcal{C}_{k-1}|,|\mathcal{D}_{k-1}| $ are $\iid$ $\Pois(4\log d )$ and \item[(iv)]   $\mathcal{|A}_1|,|\mathcal{B}_1|,|\mathcal{C}_1|,|\mathcal{D}_1|, \ldots,\mathcal{|A}_{k-1}|,|\mathcal{B}_{k-1}|,|\mathcal{C}_{k-1}|,|\mathcal{D}_{k-1}| $  are independent. 
\end{itemize}
We set $e_k=(\gamma_{k-1},v_k)$ to be some oriented edge in $G$ so that $d(v_{k},\gamma'_{k-1})=d(\gamma_{k-1},\gamma'_{k-1})-1  $ and expose a subset  $\mathcal{A}_k$ of  $\mathcal{W}_{e_{k}}(\ell+k)$ and a subset $\mathcal{C}_k$ of $\mathcal{W}_{\gamma_{k-1}}(\ell+k-1) \cap \mathcal{W}_{\gamma_{k-1}}(\ell+k)$,  disjoint of the previously exposed sets of walkers, so that $|\mathcal{A}_k| \sim \Pois(2d^{-1}\log d ) $ and  $|\mathcal{C}_k| \sim \Pois(4 \log d) $. A similar calculation as in the case of holding probability $1/(d+1)$ shows that one can construct such $(\mathcal{A}_k,\mathcal{C}_k)$. We defer the calculation to the end of the proof, as to not disrupt the flow of the argument. 

If $|\mathcal{A}_{k}|>0 $ we set $\gamma_k=v_k$. If  $|\mathcal{A}_{k}|=0 $ but  $|\mathcal{C}_{k}|>0 $ we set $\gamma_k=\gamma_{k-1}$. If $|\mathcal{A}_{k}|=0=|\mathcal{C}_{k}|$, we define $\gamma_k$ to be some vertex occupied at time $\ell +k $ by some walker in $\cup_{i=0}^{k-1}\mathcal{(A}_i \cup \mathcal{C}_i) $ of minimal distance from $\gamma_{k-1}' $. 

We define $\mathcal{B}_{k},\mathcal{D}_{k} $ and $\gamma'_k$ in an analogous manner (with $\gamma_k$ here taking the role of $\gamma'_{k-1}$ in the construction of  $\mathcal{A}_{k},\mathcal{C}_{k} $ and $\gamma_k$). Finally, note that each of the increments $d(\gamma_{k},\gamma'_{k-1}) - d(\gamma_{k-1},\gamma'_{k-1})   $ and $d(\gamma_{k},\gamma'_{k}) - d(\gamma_{k},\gamma'_{k-1}) $ is in $\{0,\pm 1 \}$ and has mean at most
\[-1 \times \Pl[|\mathcal{A}_{k}| \ge 1]+1 \times \Pl[|\mathcal{A}_{k}| =0]\Pl[|\mathcal{C}_{k}| = 0]  \]
\[=- \Pr[\Pois(2d^{-1}\log d) \ge 1]+ \Pr[\Pois(2d^{-1}\log d) = 0]\Pr[\Pois(4\log d) = 0]  \]
 \[=-(1-e^{-2d^{-1}\log d})+e^{-2d^{-1}\log d}e^{-4\log d} <-d^{-1}\log d+d^{-4}<0.\] To conclude the proof we now provide a sketch proof for the existence of  $(\mathcal{A}_k,\mathcal{C}_k)$ and  $(\mathcal{B}_k,\mathcal{D}_k)$ as above. The key calculation is that by induction, given    $e_1,f_1,\ldots,e_{k-1},f_{k-1}$, for all $i<k$ the loss to the expected number of particles at  $\gamma_{k-1}$ at time $k-1+\ell $  due to the fact we are not counting particles from $\AAA_i,\BB_i,\C_i,\DD_i $ is respectively, $2d^{-1}\log d  \times P^{k-i}(e_i^+,\gamma_{k-1}) $, $4 \log d  \times P^{k-i}(e_i^-,\gamma_{k-1})$, $2d^{-1}\log d  \times P^{k-i}(f_i^+,\gamma_{k-1}) $ and $4 \log d  \times P^{k-i}(f_i^-,\gamma_{k-1})$. Summing over these  four sets and over $i<k$, the total contribution is at most \[(8 \log d+4d^{-1}\log d)/(1-\rho_{1/2} ). \]
Thus  given    $e_1,f_1,\ldots,e_{k-1},f_{k-1}$, the the number of walkers at  $\gamma_{k-1}$ at time $k-1+\ell$ which do not belong to either of the sets  $\AAA_1,\BB_1,\C_1,\DD_1,\ldots  \AAA_{k-1},\BB_{k-1},\C_{k-1},\DD_{k-1}$ has a Poisson distribution with mean at least $\la'- (8 \log d+4d^{-1}\log d)/(1-\rho_{1/2} ) \ge 8 \log d  $. The existence of $(\AAA_k,\C_k) $ now follows from Poisson thinning. The proof of the existence of $(\BB_k,\DD_k) $ is analogous.  \qed   

\section{The $d$-regular tree - Proof of Theorem \ref{thm: darytree}.}
\label{s:tree}
Let us first explain the main idea behind the proof of Theorem \ref{thm: darytree}. As explained below, the lower bound on $\lac$ follows from Theorem \ref{thm: easylower}. So our goal is to sketch the proof that for some $C,p>0$, when $\la \ge C  \sqrt{d}$ we have that for all $u,v \in V $ we have that $u \siminfty v $ w.p.\ at least $p$. We now sketch a  construction from which we deduce  that with positive probability there are infinitely many times $t$ at which  $u $ is visited  by some walker $w$ which is at time $t$ in the friend cluster of some walker in $\W_u$, and $t$ is the first time that $w$ visits $u$. 

With slightly more care, in the proof below we manage to perform a small modification of the construction, and deduce that in fact w.p.\ at least $p$ there are infinitely many times $t$ as above at which we have that in addition    $v $ is visited  by some walker $w'$ which is at time $t$ in the friend cluster of some walker in $\W_v$, and $t$ is the first time that $w'$ visits $v$. Clearly, on this event $\as$ $u \siminfty v $ (as at each such time $t$ we get two new walkers at $u$ and $v$, and these pair of walkers have some probability of meeting each other).     

For simplicity assume that $d:=2 \ell + 1$ is odd and that $\ell \ge 2$. Set $u$ as the root of $\T_d$. We say that a child of $u$ is a left child if it is one of the $\ell+1$ leftmost children of $u$ and otherwise it is a right child. Similarly, for $z \neq u$ we say that a child of $z$ is a left child if it is one of the $\ell$ leftmost children of $z$ and otherwise it is a right child. Let $\TT $ be the induced tree on $u$ and the vertices which are right children and the path between them and $u$ contains only right children (apart from $u$). 

Observe that $\TT $ is an $\ell$-ary tree. For every site $z$ in $\TT$ we may look at the subtree $\TT_z$ containing $z$, its left children and all of their descendants. The number of walkers whose initial position lie in $\TT_z$ to reach $z$ for the first time at some time $t$, denoted by $Z_z(t)$,  can be shown to have a Poisson distribution with parameter at least $c \la$. Moreover,  for different times we have independence, by Poisson thinning. Moreover, as the trees $\TT_z $ are disjoint for different $z$'s in $\TT$ we see that $\mathbf{Z}_{z}:=(Z_z(t):t \in \Z_+) $ are independent for different $z$'s. (This follows from the requirement that the initial position of the walker is in $\TT_z$.)

Now, one scenario in which $u$ is occupied at time $2t$ by a walker belonging at time $2t$ to the friend cluster some walker in $\W_u$ is that for some path $(u_0=u,u_1,\ldots,u_t)$ in $\TT$ we have that for all $i \in \{0,\ldots,t-1 \}$,  $Z_{u_i}(i)>0 $ and one of the corresponding walkers moved from $u_i$ to $u_{i+1}$ at time $i+1$, while  for all $i \in \{1,\ldots,t \}$,  $Z_{u_i}(2t-i)>0 $ and one of the corresponding walkers moved from $u_i$ to $u_{i-1}$ at time $2t-i+1$.

As on each edge $\{u_i,u_{i+1} \}$ we have two independent requirements, each occurring w.p.\ at least $\Pr[\Pois(c \la/(d+1) )>0] $, we get that if $\la  \ge C \sqrt{d}$ for some sufficiently large $C$, then we can lower-bound the  probability that such a path  $(u_0=u,u_1,\ldots,u_t)$ as in the previous paragraph exists in $\TT$, by the probability of the event that the cluster of  $u$ in a Bernoulli bond percolation on $\TT$ (which is an $\ell$-ary tree) with parameter, say $2/(\ell-1)$, contains some vertex at distance $t$ from $u$. This probability is at least the probability that $u$ is in an infinite open cluster, which is positive. 

The difficulty is that we seek to argue that with positive this happens for infinitely many $t$'s. However, this strengthening of the previous conclusion requires only a few simple observations concerning Bernoulli percolation on trees, which we defer for the proof of Theorem \ref{thm: darytree}.    

\emph{Proof of Theorem \ref{thm: darytree}.}
The lower bound on $\lac$ follows from Theorem \ref{thm: easylower} and the fact that the spectral radius of SRW on $\mathbb{T}_d$ is $\rho(\mathbb{T}_d)= \frac{2 \sqrt{d-1}}{d}$ (cf.~\cite[Theorem 6.10]{lyonsbook}) and so by \eqref{e:rhop} the spectral radius of LSRW with holding probability $1/(d+1)$ on $\mathbb{T}_d$ is $\frac{1}{d+1}+\frac{d}{d+1}\frac{2 \sqrt{d-1}}{d} $. We now prove the upper bound. By Theorem \ref{thm: mainupper} we may assume that $d \ge 4$. Fix some $u,v \in V$. We shall show that if $\la \ge C \sqrt{d} $ for some absolute constant $C$ to be determined later, then $\Pr_{\la}[\fc(u)=\fc(v) \mid u,v \in \Xi ]>c_{1}$ for some constant $c_{1}=c_{1}(d)>0$ independent of $(u,v)$.

Throughout the proof we condition on the event that $u,v \in \Xi$. We now set $u$ to be the root of $\T_d$. This induces a partial order $\le$, where $a \le b$ iff the path from $b$ to $u$ goes through $a$. The children of $a \in V$ are $\{b:d(a,b)=1, a \le b  \}=\{b \sim a :d(b,u)=d(a,u)+ 1   \}$ (where $d(\bullet ,\bullet)$ is the graph distance). Denote $\ell := \lfloor (d-1)/2 \rfloor  $. 

For each $a \in V $ we distinguish between its  $\ell$ leftmost children, denoted by $L_a$, and its $d-\ell-1$ rightmost children $R_a$ (apart from $a=u$, for which $R_{u}$ is taken to be the $d-\ell $ rightmost children of $u$). Let $\mathcal{R}$ be the collection of all vertices such that the path between them and $u$ is contained in $\{u\} \cup (\cup_{a \in V}R_a)  $. By symmetry, we may assume that $v \in \RR$.

 For each $a \in \mathcal{R} $ we denote by $\TT_{a,L} $ the tree rooted at $a$ with vertex set \[V_a:=\{a\} \cup (\cup_{b' \in L_a} \{b : b' \le b  \})\] (where $\TT_{a,L}$ is the induced graph on this set; in other words, this is the tree containing $a$ and its left children, along with all of their descendants). For each $a \in \RR $ and $t \ge 1$ let $\W_{a,L}(t)$ be the set of walkers whose initial position is in $V_a \setminus \{a\} $ that reached $a$ for the first time at time $t$. Set $\W_{a,L}(0):=\W_a$. For $a \in \RR$ and $b \sim a$ let $\W_{(a,b),L}(t)$ be the set of walkers in $\W_{a,L}(t)$ whose location at time $t+1$ is $b$ (\textit{i.e.}, this is the set of walkers whose initial position is in $V_a \setminus \{a\} $,  who reached $a$ for the first time at time $t$ and moved to $b$ in their next step). As $\TT_{a,L}$ and $\TT_{a',L}$ are disjoint for all $a \neq a' \in \RR$ we have that $(\W_{(a,b),L}(t))_{a,b,t:\, a \in \RR, b \sim a ,t \in \Z_+}$ are disjoint.  Hence by Poisson thinning the following holds:
\begin{itemize}
\item[(1)]
 $(|\W_{(a,b),L}(t)|)_{a,b,t:\,a \in \RR, b \sim a ,t \in \Z_+}$ are independent and for each fixed $t$, we have that $(|\W_{(a,b),L}(t)|)_{a,b:\, a \in \RR,b \sim a } $ are $\iid$ $\Pois(\alpha_t)$, where by reversibility (used in the second equality) \[\alpha_t(d+1)/\la=\sum_{b \in  V_{a}}\mathbb{P}_b(T_{a}=t)= \mathbb{ P}_a(S_1,\ldots, S_t \in V_{a}\setminus \{a\}), \] where $(S_k)_{k \ge 0}$ is a LSRW (with holding probability $\sfrac{1}{d+1}$), $T_{a}:=\inf\{s:S_s=a \}$ is the hitting time of $a$ and $\mathbb{P}_b$ denotes the law of a LSRW (with holding probability $\sfrac{1}{d+1}$) started from $b$.
\end{itemize}
Thus if $C$ is taken to be sufficiently large we get that
\begin{equation}
\label{eq: alphat}
\alpha_t\ge \sfrac{\la }{d+1} \mathbb{P}_a(\{S_k:k \ge 1 \} \subseteq V_{a}\setminus \{a\})> |\log \left(1- 2/\sqrt{d - \ell -1} \right)|,
\end{equation}
where we have used the fact that $\mathbb{P}_a(\{S_k:k \ge 1 \} \subseteq V_{a}\setminus \{a\})$ is bounded from below, uniformly in $d$ and that $\ell=\lfloor (d-1)/2 \rfloor $ and so $ |\log \left(1- 2/\sqrt{d - \ell -1} \right)| \le C_0 /\sqrt{d}$.

Denote by $\TT_{u}=(U_{u},E_{u})$ and $\TT_v=(U_{v},E_{v})$ the induced graphs on 
\begin{equation}
\label{e:Uuv}
U_{u}:=\{a \in \RR:v  \nleqslant\ a \} \quad \text{and} \quad U_v:=\{a \in \RR:v \le a \},
\end{equation}
respectively.
Crucially, by construction, $U_u$ and $U_v$ are disjoint.

 For each $s \ge 0$, $a \in \RR$ and $a' \sim a$ we denote by $J_{(a,a')}(s)$ the indicator of $|\W_{(a,a'),L}(s) |>0$. Note that by \eqref{eq: alphat}
\begin{equation}
\label{eq:p1}
p:= \inf_{t > 0} (1-e^{-\alpha_t}) \ge 2/\sqrt{d-\ell -1}.
\end{equation}
By (1) we have that
\begin{itemize}
\item[(2)] The joint distribution of $(J_{(a,a')}(s))_{a,a',s:\, a \in \RR,a' \sim a,s \ge 0}$ stochastically dominates that of independent Bernoulli$(p) $ random variables.
\end{itemize} 

We say that $u$ (respectively, $v$) is \emph{good} at time $2t$ if there exists some path $(\gamma_0=u,\gamma_1,\ldots,\gamma_t) $ in $\TT_u  $ (respectively, $(\gamma_0=v,\gamma_1,\ldots,\gamma_t)$  in $\TT_v $) such that both $J_{(\gamma_i,\gamma_{i+1})}(i)=1$ and $J_{(\gamma_{i+1},\gamma_{i})}(2t-i-1)=1$, for all $0 \le i \le t-1$. We denote the indicator of $u$ (respectively, $v$) being good at time $2t$ by $Z_u(2t)$ (respectively, $Z_v(2t)$). Note that if  $u$ (respectively, $v$) is good at time $2t$ then there is some walker $w \in \fc_{2t}(u)$ (respectively, $\fc_{2t}(v)$) which reached $u$ (respectively, $v$) for the first time at time $2t$. Thus on the event that both $u$ and $v$ are good (simultaneously) for infinitely many even times, we get that $\as$ $\fc(u)=\fc(v) $. Hence in order to conclude the proof, it suffices to show that  $u$ and $v$ are good (simultaneously) for infinitely many even times with probability at least $q>0$, for some $q$ independent of $u,v$. We do so by comparison with super-critical Bernoulli bond percolation (on $\TT_v$ and $\TT_u$) which we now define.

Bernoulli bond percolation on a graph $H=(U,F)$ with density $q$ is a random graph $H_q:=(U,F_q)$ such that $F_q \subseteq F$ is defined by including in it every edge $f \in F$ independently w.p.~$q$.
Let \[p_c(H):=\inf \{q: H_q \text{ has an infinite connected component with positive probability} \}\] be the critical density for Bernoulli bond percolation on $H$. Then \[p_c(\TT_{u})=1/(d -\ell-1)=p_c(\TT_{v}).\]   
Moreover, for all $q>p_c(\TT_{u})$ we have that $\as$ $(\TT_u)_q$ satisfies that \[p_c((\TT_u)_q)=p_c(\TT_{u})/q =p_c((\TT_v)_{q})\] (where $(\TT_u)_q$ is the graph obtained from Bernoulli bond percolation with density $q$ on $\TT_u $). In fact, for every $q>\sqrt{ p_c(\TT_u)}$ we have that
\begin{itemize}
\item[(3)]
The connected component  $C_{(\TT_u)_q}(u) $ of $u$ in $ (\TT_u)_q $ is infinite with positive probability.
\item[(4)] Let $\bar \TT:=((\TT_u)_q)_q$ be the graph obtained by Bernoulli bond percolation with density $q$ on $(\TT_u)_q$. Given that $|C_{(\TT_u)_q}(u)|=\infty$, the connected component  $C_{\bar \TT}(u) $ of $u$ in $\bar \TT$  is infinite with positive probability.
\item[(5)] There exist $\beta,\delta>0 $ such that w.p.~at least $\beta$ over the choice of $(\TT_u)_q$, the graph  $(\TT_u)_q$ satisfies that $|C_{\bar \TT}(u)|=\infty $ w.p.~at least $\delta$, conditioned on $(\TT_u)_q $.  
\end{itemize}
The same applies for $\TT_v$ (with $v$ in the role of $u$ above).
Note that by \eqref{eq:p1}
\begin{equation}
\label{eq:p}
p>\sqrt{ p_c(\TT_u)}.
\end{equation}
 Let $b \in U_{u} $ (respectively, $\in U_v $) and $b' \in R_b$, where $U_{u} $ and $U_{v} $ are as in   \eqref{e:Uuv}. Denote the distance of $b$ from $u$ (respectively, $v$) by $r$. We say that the edge $\{b,b'\}$ is \emph{forward good} if $J_{(b,b')}(r)=1$ and that it is \emph{backwards good for time} $2t$ (for $t > r$) if $J_{(b',b)}(2t-r-1)=1 $. This gives raise to the following random subgraphs of $\TT_{u}=(U_u,E_u)$ and $\TT_{v}=(U_v,E_v) $: 

 Let $\tilde \TT_{u} $ (respectively, $\tilde \TT_v $) be a graph with vertex set $U_u$ (respectively, $U_v$) and edge set \[\tilde E_u:=\{e \in E_u : e \text{ is forward good} \} \quad \text{(resp. }\tilde E_v:=\{e \in E_v : e \text{ is forward good} \} \text{)}.\] Let $\tilde \TT_{u,t}:=(U_{u}, \tilde E_{u,t}) $ be the random subgraph of $\tilde \TT_u $, defined by setting $\tilde E_{u,t}$ to be the collection of all $e \in E_{u,t} $ which are backwards good for time $2t$, where $E_{u,t}$ is the set of edges in $\TT_u $ having both end-points within distance $t$ from $u$. Define $\tilde \TT_{v,t}:=(U_v,\tilde E_{u,t}) $ in an analogous manner. Note that:
\begin{itemize}
\item[(6)]  $\tilde \TT_u$ and $\tilde \TT_v$ are independent (as $U_v \cap U_v = \eset$) and (by (2))   $\tilde \TT_u $ (respectively, $\tilde \TT_v $) stochastically dominates Bernoulli bond percolation on $\TT_u$ (respectively, $\TT_v$) with parameter $p$, where $p$ is as in \eqref{eq:p1}. 
\item[(7)] The collection of random forests $(\tilde \TT_{w,t} )_{w,t:\, w \in \{u,v\},t \ge 1}$ are conditionally mutually independent, given $(\tilde \TT_u,\tilde \TT_v)$ (this follows from (2)). 
\item[(8)] Given $ \tilde E_u \cap E_{u,t} $, the joint law of $(1_{e \in \tilde E_{u,t}})_{e \in \tilde E_u \cap E_{u,t}  }$ stochastically dominates that of i.i.d.~Bernoulli $p$ random variables, where $p$ is as in \eqref{eq:p1} (this follows from (2)).
\end{itemize}
\medskip

We say that    $\tilde \TT_u $ (respectively, $\tilde \TT_v $) is $\delta$-\emph{excellent} if the connected component of $u$ (respectively, $v$) in     $\tilde \TT_u $ (respectively, $\tilde \TT_v $) is infinite and the probability that the connected component of $u$ (respectively, $v$) in a Bernoulli bond percolation on  $\tilde \TT_u $ (respectively, $\tilde \TT_v $) with parameter $p$ is infinite is at least $\delta$. Note that by (6), the event that     $\tilde \TT_{u} $ is $\delta$-excellent is independent of the event that $\tilde \TT_v $ is  $\delta$-excellent. By (3)-(6)  and \eqref{eq:p} there exist some $\beta,\delta>0$ (independent of $(u,v)$) so that $\tilde \TT_u $ and $\tilde \TT_v $ are both  $\delta$-\emph{excellent} with probability at least $\beta $.  
\medskip

By (7)-(8), conditioned on      $\tilde \TT_{u} $ and ~$\tilde \TT_v $ both being $\delta$-excellent, the conditional joint distribution of $(Z_{w}(2t))_{w,t:\, w \in \{u,v \},t>0}$ stochastically dominates that of $\iid$ Bernoulli($\delta$) r.v.'s, and so by the Borel-Cantelli Lemma indeed $\as$ $Z_{u}(2t)=1=Z_{v}(2t)$ for infinitely many $t$'s, as desired. 

Indeed, by (7) it suffices to show that $\Pr_{\la}[Z_{w}(2t)=1 \mid \tilde \TT_{w} \text{ is }\delta \text{-excellent} ] \ge \delta $,  for each $w \in \{u,v \}$ and $t>0$. By (8), for each fixed $t$, (given $\tilde \TT_u$) the (conditional) probability that $u$ is connected in $\tilde \TT_{u,t} $ to some vertex of distance $t$ from it (\textit{i.e.}, that $Z_{u}(2t)=1$) is at least the probability that the connected component of $u$ in $(\tilde \TT_u)_p $ is infinite, which by definition of the notion of  $\delta$-excellence  is at least $\delta$, given that $\tilde \TT_u$ is $\delta$-excellent (an analogous statement holds for $v$).    
\qed
\section{Proof of Theorem \ref{thm: infcluster}}
\label{s:thm3}
Before turning to the proof of Theorem \ref{thm: infcluster} let us explain our strategy. Consider the following naive exploration process. Expose the first $t$ steps of some walker $w \in \W_v$ for some $v \in V$. Let $\mathcal{G}_1$ be the  set of walkers that $w$ met by time $t$. Pick  $t=t(\la,\rho)$ so that the expectation of $|\mathcal{G}_1|$ is at least some large constant $L$ (uniformly in $v$). Then sequentially expose the first $t$ steps of the walks performed by the walkers in $\mathcal{G}_1$ and let $\mathcal{G}_2$ be the collection of walkers not in $\mathcal{G}_1 \cup \{w\} $ which met some walker from $\mathcal{G}_1$ by time $t$. Inductively, let $\mathcal{G}_{k+1}$ be the collection of walkers not in $(\cup_{i=1}^k \mathcal{G}_i) \cup \{w\} $ which met some walker from $\mathcal{G}_k$ by time $t$. 

\medskip

The problem with this naive approach is that it is not clear that for large $k$, ``typically": for $w' \in \mathcal{G}_k$ (or even for at least some fixed small fraction of $w' \in \mathcal{G}_k$) we have that the expectation of the contribution of $w'$ to $|\mathcal{G}_{k+1}|$ is large, because the contribution is restricted to walkers not in  $(\cup_{i=1}^k \mathcal{G}_i) \cup \{w\} $ (plus we need to avoid double-counting contributions of different walkers in  $\mathcal{G}_k$, corresponding to the case that two or more walkers in  $\mathcal{G}_k$   discover the same walker). 
However, as follows from our analysis below, if $\rho$ is sufficiently small (some precise version of) the statement of the previous sentence indeed holds. 

Below we consider ``$s$-walks" (defined by looking at a walk only at times which are multiples of $s$ for some sufficiently large $s=s(\rho)$) in order to obtain walks with sufficiently small spectral radius. Instead of the aforementioned naive aforementioned exploration process, we work with a variant of an exploration process due to Benjamini, Nachmias and Peres \cite{benjamini} which allows us to perform effectively the bookkeeping of which ``active but still unchecked" walkers (\textit{i.e.}, walkers already recruited to the exploration process, such that the $s$-walk performed by them is yet unexposed)  are likely to recruit ``many" new walkers to the exploration process. 
\begin{proof}
 Recall that $t_{C,\la}:=\lceil \frac{C}{\la(1-\rho)} \rceil $, where $\rho$ is the spectral radius of LSRW on $G=(V,E)$. Fix some $v \in V$. By a standard use of Kolmogorov's 0-1 law, it suffices to show that $\Pl[|\fc_{t_{C,\la}}(v)|=\infty]>0 $, provided that $C$ is sufficiently large. In particular, we may condition on $v \in \Xi$. Denote \[s:=\lceil8 K/(1-\rho) \rceil \quad \text{and} \quad M:=\lceil 32K/\lambda \rceil, \quad \text{where } K \ge 3 \] shall be determined later.  Consider the random walk obtained by replacing the transition kernel $P$ by $Q:=P^{s}$ (\textit{i.e.}, every step of this walk is $s$ steps of the original LSRW). We refer to such walks as $s$\emph{-walks} and  denote it by $(S_t^{(s)})_{t \ge 0}$ and the corresponding probability measure (for initial state $u$) by $\Pr_{u}^{(s)}$ (similarly, when the initial distribution of the walk is $\mu$ we write $\Pr_{\mu}^{(s)}$).

Our strategy is to expose a subset of $\fc_{sM}(v)$ via a variant of an exploration process due to Benjamini, Nachmias and Peres \cite{benjamini}. Recall that $\W_u(t) $ is the set of walkers which are at vertex $u$ at time $t$.  Our exploration process produces increasing sets of space-time coordinates $\{\AAA_{\ell} \}_{\ell \ge 0}$, which are subsets of $V \times \{st: 0 \le t \le M \}$ so that for all $\ell$ and all $(u,st) \in \AAA_{\ell} $ we have that $\W_u(st) \subseteq \fc_{sM}(v)$. Start with $\AAA_0:=\{(v,0) \}$. We  proceed by exposing the first $sM$ steps of the walk $(\mathbf{w}_v(i) )_{0 \le i \le s M}$  performed by some walker in $\W_v$ and set \[\mathcal{A}_1:= \{\mathbf{(w}_v(ts),ts):0 \le t \le M ,\, \mathbf{w}_v(ts) \notin \{\mathbf{w}_v(t's):t' <t \} \}\] (in simple words, these are the space time co-ordinates of the first $M$ steps of the corresponding $s$-walk, after we omit repetitions in the space co-ordinate),
$\C_1:=\{(v,0) \}$ and $\mathcal{U}_1:=\mathcal{A}_1 \setminus \C_1$. We will construct inductively sets $\mathcal{U}_{\ell},\C_{\ell}$,   $\AAA_{\ell}:=\mathcal{U}_{\ell}\cup \C_{\ell} $  and
\begin{equation}
\label{eq: Aell}
A_{\ell}:=\{u : (u,st) \in \AAA_{\ell} \text{ for some }t \} \end{equation}
such that $\AAA_1 \subseteq \AAA_2 \subseteq \cdots$. 
To avoid double-counting (which may arise since $\W_u(st)  $ and $\W_{u'}(st')$ need not be disjoint), we consider certain subsets of the $\W_u(st) $'s. Set \begin{equation}
\label{eq: Wuell}
\W_{u}^{\ell}(st):=\W_u(st) \setminus \bigcup_{a \in A_{\ell},t' \in \Z_{+}:\, (a,t') \neq (u,t) }\W_{a}(st').
\end{equation}
That is $\W_{u}^{\ell}(st)$ is the collection of walkers occupying $u$ at time $st$ which avoid $A_{\ell}$ throughout their $s$-walks, apart from at time $t$ of the $s$-walk (that is, they did not visit any $a \in A_{\ell} $ at any time in $s \Z_+ \setminus \{st \} $, where $s \Z_+:=\{sz:z \in \Z_+ \}$).
 
From the construction below it will be clear that for all $\ell$ \[|\AAA_{\ell}|=|A_{\ell}| \quad \text{and} \quad |\C_{\ell}|=\ell . \]  At each stage $\ell$ some of $(u,st) \in \AAA_{\ell}$ will be \emph{checked}, $\C_{\ell} $, and some \emph{unchecked}, $\mathcal{U}_{\ell} $. As long as $\mathcal{U}_{\ell} $ is non-empty we can proceed with the $(\ell+1)$-th stage, in which we pick some $(u,st) \in \mathcal{U}_{\ell}$ (the manner in which we choose $(u,st)$ shall be described later) and first expose $|\W_{u}^{\ell}(st) |$ and set  $\mathcal{C}_{\ell+1}:=\mathcal{C}_{\ell} \cup \{ (u,st)\} $. If $|\W_{u}^{\ell}(st) |=0 $ we set  $\mathcal{U}_{\ell+1}:= \mathcal{U}_{\ell} \setminus \{(u,st) \} $.   Otherwise, we pick one walker $w$ from $\W_{u}^{\ell}(st)$ and expose its walk by time $sM$, $(\mathbf{w}(i))_{0 \le i \le sM} $ and set  $$\mathcal{U}_{\ell+1}:=( \mathcal{U}_{\ell} \cup \{(\mathbf{w}(is),is): 0 \le i \le M,\, i \neq t,\, \mathbf{w}(is) \notin \{\mathbf{w}(js):j < i \}  \} )\setminus \{(u,st) \}$$
(in simple words, we add to  $\mathcal{U}_{\ell}$ some of the space time co-ordinates of the first $M$ steps of the $s$-walk of $w$, where we avoid taking more than one pair with the same space co-ordinate, and then subtract from it $\{(u,st) \}$).  
We conclude the step by setting $\AAA_{\ell+1}:= \mathcal{C}_{\ell+1} \cup \mathcal{U}_{\ell+1}$. To motivate what comes, assume for the moment that we can pick $(u,st) \in \mathcal{U}_{\ell}$ such that
\begin{equation}
\label{eq: goodust}
\Pr^{(s)}[ \forall \, t' \neq t, \quad S_{t'}^{(s)} \notin A_{\ell} \mid S_t^{(s)}=u ] \ge 1-2e^{-4K}.
\end{equation}
From the analysis below and Poisson thinning, it follows that for such $(u,st)$ we have that \[|\W_{u}^{\ell}(st)|>0  \text{ w.p.~at least }q:= 1-\exp [-\lambda(1-2e^{-4K})] \text{ and} \] \[\mathbb{E}[|\mathcal{U}_{\ell+1}|-| \mathcal{U}_{\ell} | \mid \AAA_{\ell},|\W_{u}^{\ell}(st)| ] \ge M/4 \text{ on the event } |\W_{u}^{\ell}(st)|>0 \] (it equals $-1$ on the complement). As $\la \in (0,1]$,  provided that $K$ is sufficiently large, in such stage  \[\mathbb{E}[|\mathcal{U}_{\ell+1}|- \mathcal{|U}_{\ell} |] \ge q M/4 -(1-q)\ge 8q  K/\lambda  -1\ge 4K .\] If we could always pick such $(u,st)$, then it is intuitively clear that with positive probability $|\mathcal{A}_{\ell}| \ge 2\ell $ for all $\ell$ and thus the construction will have infinitely many stages, implying the desired result. As we now explain, at least a $(1-e^{-4K}) $-fraction of $(u,st) \in \AAA_{\ell}$ satisfy \eqref{eq: goodust}, and thus as long as  $|\mathcal{A}_{\ell}| \ge 2\ell $, we will indeed be able to choose  $(u,st) \in \mathcal{U}_{\ell}$ satisfying \eqref{eq: goodust}. 
\medskip

Following \cite{benjamini}, given some $A \subseteq V$ and $\alpha \in (0,1) $ we say that $a \in V$ is $(A,\alpha)$-\emph{good} if $\Pr_{a}^{(s)}[T_A^+<\infty] \le \alpha$, where $T_A^+:=\inf \{t>0:S_t^{(s)} \in A \} $.
Denote the uniform distribution on $A$ by $\pi_A$. As the spectral radius of $Q=P^{s}$ is $\rho^s \le e^{-8K}$, it follows from Lemma 2.1 in \cite{benjamini} that for every finite $A \subset V $
\begin{equation}
\label{eq: alphagood0}
\Pr_{\pi_{A}}^{(s)}[T_A^+< \infty] \le \rho^s  \le e^{-8K}. 
\end{equation}
It follows from \eqref{eq: alphagood0} that for every finite $A \subset V $, the set \[G_{A}:=\{a \in A: a \text{ is }(A,e^{-4K})\text{-good} \}\] satisfies that 
\begin{equation}
\label{eq: alphagood}
|G_{A}|/|A| \ge 1-e^{-4K}. 
\end{equation}
 Fix some $a \in G_A$ and $k \le M  $. Let \[\W_{a}(A,ks):=\W_a(ks)\setminus \cup_{\ell \in \Z_+,\, a' \in A:\,(\ell,a')\neq (k,a) }\W_{a'}(\ell s) \] be the collection of walkers which are at vertex $a$ at time $ks$, which avoid $A$ throughout their $s$-walks, apart from at time $ks$ (time $k$ of their $s$-walk). Note that when we take $A=A_{\ell}$ and $a \in A_{\ell} $, we have that $\W_{a}(A,ks)=\W_{u}^{\ell}(ks) $ (where $A_{\ell}$ and $\W_u^\ell(ks)$ are as in \eqref{eq: Aell}-\eqref{eq: Wuell}). This  allows us to translate the conclusion below into one concerning \eqref{eq: goodust}.

Observe that by reversibility if  $a \in G_{A}$ and $\mathbf{(w}(t))_{t \ge 0} $ is the walk performed by some walker $w \in \W_{a}(A,ks)  $, then the walks $\mathbf{(w}_{\mathrm{forward}}(t))_{t \in \Z_+}:= \mathbf{(w}((k+t)s))_{t \in \Z_{+}}$ and $\mathbf{(w}_{\mathrm{backward}}(t))_{0 \le t \le k}:=\mathbf{(w}((k-t)s))_{0 \le t \le k} $ are (independent) $s$-walks conditioned to avoid $A$, apart from at time 0. In particular,  \eqref{eq: goodust} holds for $A$ in the role of $A_{\ell}$ as  $a \in G_A$. Again using $a \in G_A$ we have that  \[\mathbb{E}[|\W_{a}(A,ks)|] \ge \mathbb{E}[|\W_{a}(ks)|](1-2\Pr_{a}^{(s)}[T_A^+<\infty]) \] \[ \ge  (1-2e^{-4K})\mathbb{E}[|\W_{a}(ks)|]=\la (1-2e^{-4K}) .\]
By Poisson thinning if $a \in G_{A}$, then for all $k$ we have that $|\W_{a}(A,ks)| $ has a Poisson distribution with mean at least $\la (1-2e^{-4K}) $.

Using \eqref{eq: spectralxy} it is not hard to show that the expected number of times an $s$-walk of length at most $M$ intersects itself is at most $M \rho^s/(1-\rho^s) \le \sfrac{9}{10} M e^{-8K} \le e^{-4K}/\la$, provided that $K$ is sufficiently large. Thus by Markov's inequality, if   $a \in G_{A}$ and $\mathbf{(w}(t))_{t \ge 0} $ is the walk performed by some walker $w \in \W_{a}(A,ks)  $ for some $k \le M$, then $\mathbf{(w}(ts))_{t:\, 0 \le t \le M,\, t \neq k}$ visits at least $M/4 \ge 2K/ \la $ distinct vertices with probability at least $p:=1-\frac{1}{Ke^{4K}(1-2e^{-4K}) } $.\footnote{The term $1-2e^{-4K}$ in the denominator is there since instead of taking $w \in \W_a(ks)$ we take $w \in \W_a(A,ks)$, which means that the law of its walk is conditioned to be in some set of walks whose probability (w.r.t.~the law of a walk of a walker in $\W_a(ks)$) is at least  $1-2e^{-4K}$.}

Let \[U_{\ell}:=\{u :(u,st) \in \mathcal{U}_{\ell} \text{ for some }t  \} \quad \text{and} \quad C_{\ell}:=\{u :(u,st) \in \mathcal{C}_{\ell} \text{ for some }t  \}.\] Assume that $|A_{\ell}| \ge 2 \ell$. Then $|G_{A_{\ell}}| \ge (1-e^{-4K})2 \ell > \ell $ and so $G_{A_{\ell}}\setminus C_{\ell}=G_{A_{\ell}}\cap U_\ell$ is non-empty (as $|C_{\ell}|=\ell$). As long as this is the case, in the $\ell$-th stage we expose some $(u,st) \in \mathcal{U}_{\ell} $ such that $u \in G_{A_{\ell}}\cap U_\ell  $, where the choice of $(u,st)  $  is made according to some prescribed order on $V \times \Z_+ $ (or simply according to the lexicographic order on the stage in which the walkers were discovered and their time coordinate). By the above analysis, provided that $K$ is sufficiently large, the probability that $|A_{\ell+1}|-|A_{\ell}| \ge 2K/ \la  $ is  at least $ qp=  (1-\exp [- \la(1-2e^{-4K})]) \times (1-\frac{1}{Ke^{4K}(1-2e^{-4K}) })  \ge \la/2 $ (for $\la \le 1$ and large $K$),   and so $$\mathbb{E}[|A_{\ell+1}|-|A_{\ell}| \mid |A_{\ell}| \ge 2 \ell ] \ge (2K/ \la)pq \ge K \ge 3 .$$
Combining this with Azuma inequality (applied to the Doob's martingale of $(|A_{\ell}|)_{\ell \ge 0}$), it is not hard to verify that with positive probability $|A_{\ell}| \ge 2 \ell $ for all $\ell$ (cf.~the proof of Theorem 1.1 in \cite{benjamini}) as desired.
\end{proof}
\section{A lower bound on $\lac$ in the non-amenable case}
\label{s:lowernonamen}  
\begin{theorem}
\label{thm: easylower}
 Let $G=(V,E)$ be an infinite connected non-amenable regular graph. Denote the spectral radius of LSRW on $G$ (with an arbitrary holding probability $p$) by $\rho$. Then the SN model on $G$ with holding probability $p$ satisfies
\begin{equation*}
\Pr_{\lambda}[\mathbf{Con}]=0, \quad \text{ for all } \lambda < \sfrac{1}{2}\left(\rho^{-1}-1 \right).
\end{equation*}
\end{theorem}
Throughout the section we fix the holding probability of the walks to be some constant $0 \le p<1$. Let $\mu_{\lambda}$ (respectively, $\nu_{\la} $) be the distribution of $1+2X_{\la}$ (respectively, $1+X_{\la}$), where $X_{\la} \sim \Pois( \lambda ) $. A 
\emph{\emph{lazy branching random walk}}  on $G$ with offspring distribution $\mu_{\la}$ started at a vertex $o$, denoted by $\mathrm{LBRW}(\mu_{\la},o) $, is defined as follows. At time $0$ there are a random number of particles distributed according to $\nu_{\la}$ which are all positioned at vertex $o$. Call the set of these particles generation number 0. The process is then defined inductively. At stage $t$ each particle $w$ belonging to the $t$-th generation performs one step of LSRW on $G$ from its position at time $t$, where steps performed by different particles are independent. Then it gives birth to a random number of particles (referred to as its offsprings) $Y_{w} \sim \mu_{\la}$, at its current position, independently of all other particles. The set of all the offspring of the particles from the $t$-th generation is defined to be the $(t+1)$-th generation.

The following interpretation of $\mathrm{LBRW}(\mu_{\la},o)$ is useful for our purposes. First, by including the previous generations as part of the current generation, we may think of the offspring distribution as being the same as that of $2X_{\la}$, where $X_{\la} \sim \Pois( \lambda )$. Equivalently, in this interpretation, a particle does not ``die" after giving birth to some offspring at a certain step, and may give birth to additional offspring in future stages (alternatively, we may view the particle as an ``offspring of itself").

We may think of each particle as giving birth to $\Pois(\la)$ ``regular particles" which then clone themselves. By reversibility, we may think of the regular particles as performing independent LSRWs, while the clones perform a LSRW moving backwards in time in the following sense. The law of LSRW started from $v$ is the same as the law of $(Y_s)_{s=0}^{\infty} $, where $Y_s:=X_{-s}$ for all $s \ge 0 $ and $(X_{s})_{s \in \Z }$ is a bi-infinite LSRW conditioned on being at $v$ at time $0$. Hence we may assume the walk of the clone particle is sampled in that manner. 

We now describe a process which, based on the previous two observations, is essentially equivalent to  $\mathrm{LBRW}(\mu_{\la},o)$. In particular, the expected total number of visits to each vertex (including multiplicities) is the same for the  two processes. While the definition of this process is somewhat cumbersome, it will be transparent that this process stochastically dominates the exploration process used below in order to ``expose" $\fc(o)$, the friend cluster of $\W_o$. We intentionally use similar notation to describe this variant of   $\mathrm{LBRW}(\mu_{\la},o)$ as the one used later in the exploration process of  $\fc(o)$.  

\medskip

In the $0$-th generation, $\mathcal{V}_{0,0} $, we start with $1+\Pois(\la)$ walkers  $v_{0,0,1},\ldots,v_{0,0,|\mathcal{V}_{0,0} |}$ at $o$. Let each $v_{0,0,j}$ perform a $\Z$-indexed (bi-infinite) random walk  $(\mathbf{v}_{0,0,j}(t))_{t \in \Z }$ on $G$, conditioned to be at $o$ at time $0$. Such a walk can be sampled by taking two independent $\Z_+$-indexed walks started at $v$, $(\mathbf{fv}_{0,0,j}(t) )_{t \ge 0}$ and $(\mathbf{bv}_{0,0,j}(t) )_{t \ge 0}$ (which can be thought of as 2 independent walks performed by 2 separate particles) and concatenating one to the reversal of the other as follows $\mathbf{v}_{0,0,j}(t):=\mathbf{fv}_{0,0,j}(t) $ and $\mathbf{v}_{0,0,j}(-t):=\mathbf{bv}_{0,0,j}(t) $ for all $t \ge 0$.

In the first stage we expose $\mathbf{v}_{0,0,j}( \pm 1 )$ for all $j$ (in the above interpretation, we expose one step of the walk of the forward particle $\mathbf{fv}_{0,0,j}(1)$ and one of the backward
 particle $\mathbf{bv}_{0,0,j}(1)$) and plant at $\mathbf{v}_{0,0,j}( \pm 1 ) $ (independently for different $j$'s and for $\pm 1 $) $\Pois(\la)$ walkers performing (independent) $\Z$-indexed random walks on $G$ conditioned to be at $\mathbf{v}_{0,0,j}( \pm 1 )$ at time $\pm 1 $, respectively. Denote the set of walkers planted at stage 1 at time $\pm 1$ by $\mathcal{V}_{1,\pm 1}=\{v_{1,\pm 1,1},\ldots,v_{1,\pm 1, |\mathcal{V}_{1,\pm 1}|} \}$, respectively. The construction continues inductively as follows: 

By the end of stage $r$, for all $0 \le i \le r $ and $-i \le j \le i $ such that $i-j$ is even, we have already defined  $\mathcal{V}_{i, j}=\{v_{i,j,1},\ldots,v_{i,j, |\mathcal{V}_{i,j}|} \}$ the set of walkers planted at stage $i $ and time $j$, and  for all  $1 \le k \le  |\mathcal{V}_{i,j}| $ exposed $(\mathbf{v}_{i,j,k}(t ))_{t: |t-j| \le r-i}$, where $v_{i,j,k}$ is the $k$-th walker in $\mathcal{V}_{i, j}$ and $(\mathbf{v}_{i,j,k}(t ))_{t \in \Z} $ is the walk she performs. In the $(r+1$)-th stage we expose for all $i,j,k$ as before $\mathbf{v}_{i,j,k}(j \pm (r+1-i) )$ and plant at  $\mathbf{v}_{i,j,k}(j \pm (r+1-i) )$ (independently for different $(i,j,k)$'s and for $j \pm (r+1-i) $) $\Pois(\la)$ walkers performing (independent) $\Z$-labeled random walks on $G$ conditioned to be at $\mathbf{v}_{i,j,k}(j \pm (r+1-i) )$ at time $j \pm (r+1-i) $, respectively. Finally, we denote the set of walkers planted at stage $r+1$ at time $\ell $ by $\mathcal{V}_{r+1,\ell}=\{v_{r+1,\ell ,1},\ldots,v_{r+1,\ell, |\mathcal{V}_{r+1,\ell}|} \}$.     

\medskip

Below we expose  $\fc(o) $ in ``slow motion" using an exploration process. At each stage $t$ of the exploration process, new walkers are ``recruited" to the friend cluster by meeting at some time $s \le t$ walkers already belonging to the exploration process.  The walkers recruited at stage $t$ can be thought of as the $t$-th generation of the exploration process. 

Let $w$ be some walker in the $t$-th generation of the exploration process who was recruited at stage $t$ due to an acquaintance which occurred at time $s$ (the set of such walkers shall be denoted below by $\W_{t,s}$).
 Instead of exposing in the $(t+1)$-th stage the entire trajectory of $w$, we expose its position at times $s+1$ and $s-1$. At stage $t+2$ we expose its position at times $s+2$ and $s-2$ (if $s \ge 2$), and so on (at stage $t+i$ we expose its position at time $s+i$ and if $s \ge i$ also at time $s-i$). 

Let $(\mathbf{w}(n))_{n \ge 0}$ be the infinite walk performed by $w$. We can think of $w$ as two separate particles, one a forward particle performing the forward walk $(\mathbf{w}(t+n))_{n \ge 0}$ and the other a clone performing the reversed walk $(\mathbf{w}(t-n))_{n:0 \le n \le t}$. At each stage, for every previously exposed walker $w$ we expose one step of its forward walk and one step of its reversed walk (or in the above terminology, one step of the walk performed by its clone), if it was not fully exposed already. The particle (or clone) recruits new walks if she meets them at the space-time coordinate of her walk which was exposed at the current stage, and if those walkers avoided all the space-time coordinates previously exposed by the exploration process (otherwise these walkers would have already been recruited to the exploration process).

Using Poisson thinning, we can dominate this exploration process by the equivalent formulation of $\mathrm{LBRW}(\mu_{\la},o)$, involving the $\Z$-valued walks and the sets $\mathcal{V}_{t,s}$. Indeed there are two differences between the two. The first is that in the latter the walks of the particles moving backwards in time continue all the way to time $-\infty$  instead of stopping at time 0.   The second difference is that in the exploration process of $\fc(o)$ each particle can only recruit ``new" walkers (and their clones), which means that these walkers have to avoid certain space-time coordinates previously exposed by the exploration process. Thus, by Poisson thinning her offspring (= new walkers recruited by her at each stage and their clones) distribution is stochastically dominated by the $2\Pois(\la)$ distribution. 

 Unfortunately, while the aforementioned stochastic domination is intuitively clear,  its proof requires some cumbersome bookkeeping and no much additional insights beyond the ones described in the above intuitive explanation. For this reason we defer the proof of Proposition \ref{prop: dominationbybrw} to Appendix \ref{a:p8.2}.

By Lemma \ref{lem: infcollide} every vertex is visited infinitely often $\as$ Thus on the event $\con \cap \{o \in \Xi \} $ (assuming it has a positive probability) we have that $\fc(o)$ (the friend cluster of $\W_o$) is the set of all walkers, and so $o$ is visited by walkers in $\fc(o)$ infinitely often $\as$ 
Note that if $\Pr_{\lambda}[\mathbf{Con}]>0 $, then there must be some $o$ such that $\Pr_{\lambda}[ \con
\cap \{o \in \Xi \}]>0 $, and so the expected number of times in which vertex $o$ is visited by walkers from $  \fc(o)$ including multiplicities (here we count also visits made by a walker $w \in \fc(o) $ at time $t$ in which $w \notin \fc_t(o)$, \textit{i.e.}, before the walker $w$ joined the friend cluster of the walkers in $\W_o$) is  infinite, as on the event $ \con
\cap \{o \in \Xi \} $ the last expectation is simply the expected number of visits to $o$ by all particles (with multiplicities; The number of such visits is $\as$ infinite and so this expectation is infinite even on the event $\con
\cap \{o \in \Xi \} $). Hence the assertion of Theorem \ref{thm: easylower} follows by combining the following proposition and lemma.
\begin{proposition}[proof deferred to Appendix \ref{a:p8.2}]
\label{prop: dominationbybrw}
Let $X_{v}$ be the number of times  vertex $v$ was visited by a walker from $\fc(o)$ (including multiplicities) when the density of the walkers is taken to be $\la$.  Let $Y_{v}$ be the number of times that vertex $v$ was visited by a particle in  $\brw$ (where if a particle in the lazy branching random walk   $\brw$  is born at vertex $v$ this also contributes to $Y_v$). Let $\nu_1$ and $\nu_2$ be the  laws of $(X_v)_{v \in V }$ and $(Y_v)_{v \in V }$, respectively. Then $\nu_2$ stochastically dominates $\nu_1$.
\end{proposition}
\begin{lemma}
\label{lem: triansienceofBRW}
For $v \in V$ and $n \ge 0$, let $Q_{n}(v)$ be the the number of particles belonging to the $n$-th generation of $\brw$ which were born at vertex $v$. Then for all $v$ and $n \ge 1 $,
\begin{equation}
\label{eq: transiencebrw1}
\mathbb{E}[Q_{n}(v)] = (1+\la)(1+ 2 \lambda )^{n-1}  P^{n}(o,v) \le  [(1+2
\lambda) \rho]^{n}.
\end{equation}
In particular, if $\la < \half ( \rho^{-1} -1) $ we have that $\sum_{n=0}^{\infty} Q_n(o)<\infty $ $\as$
\end{lemma}
The proof of the equality in (\ref{eq: transiencebrw1}) is obtained by a simple induction on $n$, performed simultaneously over all vertices (we omit the details). The inequality in (\ref{eq: transiencebrw1}) follows from (\ref{eq: spectralxy}). We note that it is shown in \cite{criticalBRW} that the critical mean offspring distribution for a branching random walk is  $1/\rho$ and that a critical branching random walk is transient (i.e., it $\as$ visits every vertex only finitely many times). Hence if $(1+2
\lambda) \rho=1$ then the $\brw$ is transient.

\section{Concluding remarks}
\label{s:concluding}

\subsection{Refined lower bound when the holding probability is 1/2}
\label{s:refined}
In this subsection we give a rough sketch of a proof of the following theorem. 
\begin{theorem}
\label{thm: hardlower}
 Let $G=(V,E)$ be an infinite connected non-amenable regular graph. Denote the spectral radius of SRW on $G$ by $\rho$. Then the SN model on $G$ with holding probability $1/2$ satisfies
\begin{equation}
\label{e:hardlower}
\Pr_{\lambda}[\mathbf{Con}]=0, \quad \text{ for all } \lambda \text{ such that } 1+2\la(1+2e^{\la/2} ) \le  \rho^{-1}.
\end{equation}
\end{theorem}
Note that while in \eqref{e:hardlower} we are considering the SN model  with holding probability $1/2$, the term $1/\rho$ is defined w.r.t.\ SRW. For instance, for the $d$-ary tree this shows that $\lac  \ge c \log d $ when the holding probability is $1/2$, whereas in this case Theorem \ref{thm: easylower} yields a weaker lower bound which does not diverge as $d \to \infty $. Combining Theorem \ref{thm: hardlower} with Theorem \ref{thm: mainupper} yields the following.

\begin{corollary}
There exist absolute constants $c,C>0$ such that for all $d \ge 3$ we have that $c \log d \le \lac(\T_d) \le C \log d $ for SN model with holding probability is $1/2$. The same holds for every connected infinite $d$-regular Ramanujan graph, with the same $c$ and $C$.      
\end{corollary}

The reason we provide here a much less  detailed analysis than in \S\ref{s:lowernonamen} (and Appendix \ref{a:p8.2}) is that the ideas here are extremely similar to those from  \S\ref{s:lowernonamen}. Like in Appendix \ref{a:p8.2}, in order to rigorously justify the claim that the below exploration process for $\fc(o)$  is indeed dominated by the branching random walk described below, one can introduce ``dummy particles". This is meant to justify the following fact that is used implicitly below: 
\begin{itemize}
\item The  number of walkers at vertex $v$ at time $t$ which avoid a certain collection of space-time co-ordinates $(u_1,t_1),\ldots,(u_r,t_r) $ (where $u_1,\ldots,u_r \in V$ and $t_1,\ldots,t_r \in \Z_+ $, possibly $t_i>t$ for some $i$'s) is independent of  $\mathcal{X}:=(|\W_{t_i}(u_i)|:i \in [r]) $, where $|\W_{t_i}(u_i)|$ is the number of walkers at vertex $u_i$ at time $t_i$.
\item Moreover, it  is  stochastically dominated by the Poisson($\la$) distribution.
\item Furthermore, for each path $(v_{0},v_{1},\ldots,v_{s})$  that avoids the above space time co-ordinates, in the sense that for all $i \le s$ we have that $v_{i} \notin \{u_k:t_k=i \} $,  we have that the number of walkers which performed this path is independent of $\mathcal{X}$.     
\end{itemize}
 However, in order to facilitates analysis analogous to the one of Appendix \ref{a:p8.2}, the notation and bookkeeping  required here are much more cumbersome compared to the already cumbersome notation from  \S\ref{s:lowernonamen}. For the sake of clarity of presentation, we chose to present the exploration process below using as little notation as possible, and to leave it to the reader to verify the details of the claimed stochastic domination.

The idea of the proof is to explore $\fc(o)$ in an ``ultra slowed down" fashion which exploits the laziness of the walks. The exploration process below is still be dominated by a branching random walk, but in a much less wasteful fashion than as  in the proof of Theorem \ref{thm: easylower}. 

Consider the case that a walker $x$ jumps at time $t_x$ to some site $v$ from some neighboring site $u$, and that $x$ left $v$ at time $t_x'+1$. The walkers $\W_{1} $ she met at $v$ during $[t_x,t_x']$   must all be in $\fc(x)$. Each walker $w \in \W_1 $ entered $v $ at some time $t_w$ and left it at time $t_{w}'+1$ such that $[t_w,t_w'] \cap[t_x,t_x'] = \eset $. Let $\W_2$ be the collection of walkers $z$ not belonging to $\{x\} \cup \W_1 $ such that they entered $v $ at some time $t_{z} \ge 0 $ and left at time $t_{z}'+1$ such that $[t_{z},t_{z}'] \cap (\bigcup_{w \in \W_1}[t_w,t_w'] ) \neq \eset $. 

We can continue defining $\W_i $'s in these fashion inductively until the first $i_0$ such that $\W_{i_0+1}=\eset$. Namely, if $\W_i \neq \eset$  let $\W_{i+1}$ be the collection of  walkers $z$ not belonging to $\{x\} \cup \bigcup_{j =1}^i \W_j $ such that they entered $v $ at some time $t_{z} \ge 0 $ and left at time $t_{z}'+1$ such that $[t_{z},t_{z}'] \cap(\bigcup_{w \in \W_i}(t_w,t_w']) $.  

Clearly, $\bigcup_{j=1}^{i_0}\W_j $ must all be in $\fc(x)$. For each walker   $w \in \bigcup_{j=1}^{i_0}\W_j $ we can now reveal (``backwards step") from what vertex did it jump to $v$ (provided $t_w>0$) and to which vertex it jumped to when leaving $v$ (``forward step"). Each such walker $w$ starts in its forward and backwards step a new process with the same description as $(\W_i)_{i=1}^{i_0}$ above. However, at each stage we wish to not count walkers already recruited to the exploration process at previous stages (or earlier on at the same stage).   

As in \S\ref{s:lowernonamen}, for each particle recruited to the exploration process we will expose at each stage its trajectory one step forward and one step backwards. However, one crucial difference is that now we reveal its \emph{non-lazy} trajectory. By this, we mean the following. The  non-lazy trajectory corresponding to a SRW trajectory $(u_0,u_1,\ldots)$ is obtained by deleting consecutive repetitions. That is, it is $(v_0,v_1,\ldots)$  where $v_i:=u_{\tau_i}$ and $\tau_i:=\inf \{j>\tau_{i-1} :u_{j} \neq u_{\tau_{i-1}} \} $. 

Let $x $ be a particle recruited to the exploration process at some stage $k$. Let $(v_0,v_1,\ldots)$ be its non-lazy trajectory. Assume that $w$ was recruited at location $v_m$ during the time interval $[\tau_m,\tau_{m+1}-1] $ (with $\tau_{\cdot}$ as above).    At a stage $i>k$ we reveal (forward step) $v_{m+i-k} $ and if $i-k \le m$ also  $v_{m-(i-k)} $ (backwards step). We can then define $\mathcal{U}_0 $ to be the collection of particles $y$ not previously recruited to the exploration process, that  jumped to $v_{m+i-k} $ at some time $t_y \ge 0$ and stayed there until time $t_y'+1 $ so that $[t_y,t_y'] \cap [\tau_{m+i-k},\tau_{m+i-k+1}-1] \neq \eset $.   Recruit the walkers from   $\mathcal{U}_0  $ to the exploration process.    Let  $\mathcal{U}_1  $ be the collection of walkers $y$ not previously recruited to the exploration process, who   jumped to $v_{m+i-k} $ at some time $t_y \ge 0$ and stayed there until time $t_y'+1 $ so that $[t_y,t_y'] \cap (\bigcup_{z \in \mathcal{U}_1} [t_{z},t_{z}']) \neq \eset $. Recruit the walkers from   $\mathcal{U}_1  $ to the exploration process. We can continue defining $\mathcal{U}_{j+1}$ inductively in this fashion as  long as $\mathcal{U}_j \neq \eset $. Let $j_0 $ be the minimal integer such that $\mathcal{U}_{j_0+1}=\eset $. Then the collection of particles recruited by $x$ at stage $i$ via its forward step is $\bigcup_{j=0}^{j_0}\mathcal{U}_j $.

We now define the collection of particles recruited by $x$ at stage $i$ via its backwards step. We assume that $i-k \le m$ as otherwise there is no such backwards step.  Let  $\mathcal{B}_0 $ to be the particles $y$ who  jumped to $v_{m-(i-k)} $ at some time $t_y \ge 0$ and stayed there until time $t_y'+1$ such that   $[t_y,t_y'] \cap [\tau_{m-(i-k)},\tau_{m-(i-k)+1}-1] \neq \eset $ and have not been previously recruited to the exploration process. Recruit the walkers from   $\mathcal{B}_0  $ to the exploration process. Let  $\mathcal{B}_1  $ be the collection of walkers $y$ not previously recruited to the exploration process, who   jumped to $v_{m-(i-k)} $ at some time $t_y \ge 0$ and stayed there until time $t_y'+1 $ so that $[t_y,t_y'] \cap (\bigcup_{z \in \mathcal{B}_1} [t_{z},t_{z}']) \neq \eset $. Recruit the walkers from   $\mathcal{B}_1  $ to the exploration process. We can continue defining $\mathcal{B}_{j+1}$ inductively in this fashion as  long as $\mathcal{B}_j \neq \eset $. Let $j_0' $ be the minimal integer such that $\mathcal{B}_{j'_0+1}=\eset $. Then the collection of particles recruited by $x$ at stage $i$ via its backwards step is $\bigcup_{j=0}^{j_0'}\mathcal{B}_j $.  

As in \S\ref{s:lowernonamen} at each stage we reveal the backwards and forward steps of all recruited particles sequentially according to some predetermined order. This affects the notion of ``not being previously recruited to the exploration process" used above (during each stage this notion is updated as the stage progresses). Moreover, in order to be at $\mathcal{U}_i$ (respectively, $\mathcal{B}_i$) we require a walker to not be in $\mathcal{U}_j$ (respectively, $\mathcal{B}_j$) for all $j<i$).

As mentioned above, we shall dominate this exploration process via a branching random walk. The offspring distribution of this branching random walk has the same law as $1+2W$, where $W$ has a rather complicated law we shall soon describe. The source of the $+1$ term and of the multiplicative term $2$ is exactly the same as in \S\ref{s:lowernonamen} (particles don't die explains the term  $+1$ , and the fact each paarticlllle progresses in both directions of time explains the term 2). We  seek to take the law of $W$ to be one which dominates the laws of  $\bigcup_{j=0}^{j_0}\mathcal{U}_j $ and  $\bigcup_{j=0}^{j_0'}\mathcal{B}_j $ described above.

To do so, it is useful  to describe the evolution of   $\mathcal{U}:=\bigcup_{j=0}^{j_0}\mathcal{U}_j $  one time unit at a time,  from $\tau_{m+i-k}$ to $\max_{w \in \mathcal{U} }t'_w $  (rather than one index at a time, from $\mathcal{U}_0$ to $\mathcal{U}_{j_0}$; a similar description applies to   $\bigcup_{j=0}^{j_0'}\mathcal{B}_j $). However, we also need to consider its evolution backwards in time (which takes place between time $\min_{w \in \mathcal{U} }t_w$ and $\tau_{m+i-k}$), as some walkers in $\mathcal{U}_0$ could have been at $v_{m+i-k}$ both at time  $\tau_{m+i-k}$ and at time  $\tau_{m+i-k}-1$.

 Moving forward in time, each particle stays in  $v_{m+i-k}$ with probability $1/2$. By Poisson thinning, the number of new (\textit{i.e.}, not previously recruited) particles to jump to   $v_{m+i-k}$ at each time is stochastically dominated by the $\Pois(\la/2) $ distribution. 

 For the evolution forward in time, we are interested in the number of the walkers recruited between time $\tau_{m+i-k}$ and $\max_{w \in \mathcal{U} }t'_w $. The latter is the first time $t > \tau_{m+i-k} $ at which no particles that were in  $v_{m+i-k}$ at time $t-1$ stayed at $v_{m+i-k}$ at time $t$. If we reverse time, the same description  is valid backwards in time - that is, provided some walkers in $\mathcal{U}_0$  were  at $v_{m+i-k}$ both at time  $\tau_{m+i-k}$ and at time  $\tau_{m+i-k}-1$,  we are looking at the maximal time $t < \tau_{m+i-k} $ at which there are no particles at $v_{m+i-k}$ that were there also at time $t+1 $.

 Consider the Markov chain $(X_t)_{t \ge 0 }$ that at time $t+1$ evolves to $X_{t+1}=Y_{t+1}+Z_{t+1} $, where $Z_{1},Z_2,\ldots$ are i.i.d.\ $\Pois(\la/2)$, and given $X_t$ we have that $Y_{t+1} $ has a $\mathrm{Bin}(X_t,1/2)$ distribution and is independent of $Z_{t+1}$. We extend this process to a bi-infinite process by setting for all $t \ge 0$, $X_{-t-1}=Y_{-t-1}+Z_{-t-1} $, where $Z_{-1},Z_{-2},\ldots$ are i.i.d.\ $\Pois(\la/2)$, and given $X_{-t}$ we have that $Y_{-t-1} $ has a $\mathrm{Bin}(X_{-t},1/2)$ distribution and is independent of $Z_{-t-1}$. Let \[\sigma=\inf \{t:Y_{t+1}=0 \} \quad \text{and} \quad  \sigma'=\inf \{t:Y_{-t-1}=0 \} .\] We consider the case that $X_0 \sim 1+ \Pois(\la) $. It is not hard to see that by reversibility we can take \[W=(X_{0}-1)+\sum_{i=1}^{\sigma}Z_i+\sum_{i=-\sigma'}^{-1}Z_i .\] 

By abuse of notation, if $\xi \sim \Pois(a)$ then we refer to the law of $1+\xi$ as $1+\Pois(a) $. Using similar reasoning as in footnote 1, we argue that given $\sigma>1 $, we have that $Y_1 $ is stochastically dominated by the law $1+ \Pois(\la/2)$.  Indeed, we may think of $Y_1$ as the number of successes in $X_0$ Bernoulli($1/2$) trials. We are interested in the conditional law of $Y_1$ given $Y_1>0$. If the first trial is a success, then the conditioning on $Y_1>0$ does not affects the number of successes in the remaining $ \Pois(\la)$ trials, and so by Poisson thinning the law of the number of additional successes has the $\Pois(\la/2)$ distribution. 

If the first trial is a failure, then the conditional law of the total number of success is, again by Poisson thinning, the law of a $\Pois(\la/2)$ random variable conditioned on being positive. As in footnote 1, by considering the number of arrivals in $[0,1] $ in a rate $\la$ Poisson process, and conditioning on the location of the first arrival, we see that the aforementioned law is stochastically dominated by the $1+\Pois(\la/2) $ distribution.     
 
 It follows by induction that given $\sigma>t $, we have that $X_t$ is stochastically dominated by the $1+\Pois(\la)$ distribution. Hence $\sigma $ is stochastically dominated by the Geometric distribution with parameter $1/(2e^{\la/2}) $ (which is the probability that  $1+\Pois(\la)$ independent Bernoulli($1/2$) trials all fail). Likewise, the same applies to $\sigma' $ by reversibility. By Wald's equation we have that $\mathbb{E}[W]=\la(1+2e^{\la/2} )$. As in \S\ref{s:lowernonamen}, the condition $1+2\mathbb{E}[W] \le  1/\rho $ implies the  branching (simple) random walk on $G$ with offspring distribution $1+2W$ is transient, which as in  \S\ref{s:lowernonamen}, can be used to argue that the above exploration process for $\fc(o)$ $\as$ does not visit all vertices. This concludes the sketch of the proof of Theorem \ref{thm: hardlower}.

\subsection{Improving the dependence on the distance of the spectral-radius from 1}
\label{s:improving}
As we now explain, with a bit more care, the terms $\frac{2}{1-\rho}  $ and  $\frac{20}{1-\rho_{1/2}}$ from Theorem \ref{thm: mainupper} can be replaced by  $\frac{c_{1}}{\sqrt{1-\rho}}  $ and  $\frac{c_{2}}{\sqrt{1-\rho_{1/2}}}$, respectively, for some  constants $c_1,c_2>0$.  Similarly, in Theorem \ref{thm: infcluster} we could have taken $t_{C,\la}$ to be $\lceil \frac{C}{\la \sqrt{1-\rho}} \rceil$, rather than $\lceil \frac{C}{\la (1-\rho)} \rceil$. 

Let $P$ be the transition kernel of SRW or lazy SRW with holding probability $p \le 1/2 $  on an infinite connected regular graph $G=(V,E)$. Let $\rho(P)$ be the spectral-radius of $P$. By inspecting the proofs of Theorem \ref{thm: mainupper} and Theorem  \ref{thm: infcluster}, such improvements can be derived from the estimate   
\begin{equation}
\label{e:improvement}
 \sum_{t=0}^{\infty} \sup_{x,y} P^{t}(x,y) \le C_0/ \sqrt{1-\rho(P)},
\end{equation}
 rather than the estimate     $\sum_{t=0}^{\infty} \sup_{x,y} P^{t}(x,y) \le C_1/ (1-\rho(P))$ that we use. 

Similarly to (3) from \S\ref{sec: gap}, for all $s,t \ge 0$ and all $x \in V $ we have that \[P^{2t+2s}(x,x)=\langle P^{2t+2s}1_x,1_x \rangle=\| P^{t+s}1_x\|_{2}^2 \le \rho(P)^{2s}\| P^{s}1_x\|_{2}^{2}=\rho(P)^{2s}P^{2t}(x,x).   \] 
 
    Combining the above with the fact that  $\max_{x,y} P^{t}(x,y) \le \sup_{x}P^{2 \lfloor t/2 \rfloor}(x,x)$ (Proposition \ref{p:decayofreturn}),  yields that  \[ \sum_{t=0}^{\infty} \sup_{x,y} P^{t}(x,y)\le 2 \sum_{t=0}^{\infty} \sup_{x} P^{2t}(x,x) \le \frac{2e}{e-1}  \sum_{t=0}^{\lceil 1/(1-\rho(P)) \rceil} \sup_x P^{2t}(x,x) .\] Finally, we obtain \eqref{e:improvement}  using the fact that there exists an absolute constant $C>0$ such that   $\sup_x P^{t}(x,x) \le \frac{C}{\sqrt{t+1}}$ for all $t$ (the same constant $C$ works for all $p \le 1/2$ and all graphs $G$ as above, \textit{e.g.},\ \cite{lyonsev} -- see the discussion in the proof of Lemma \ref{lem: regen}).

\section*{Acknowledgements}
We are grateful to Itai Benjamini for suggesting the problems studied in this paper and thank him and Gady Kozma for  helpful discussions.

...

\nocite{}
\bibliographystyle{plain}
\bibliography{SN}

\appendix

\section{Appendix A: Proof of the regeneration Lemma \& Lemma \ref{lem: infcollide}}
\label{a:regen}
\emph{Proof of Lemma \ref{lem: regen}.}
The independence and the fact that the marginal distributions are Poisson follow from Poisson thinning.  Denote $P(v,A):=\sum_{a\in A}P(v,a) $. By reversibility, $\mathbb{E}_{\la}[Y_{v, A}(t)]=\la P^{t}(v,A)$ and so $$\la- \mathbb{E}_{\la}[Y_{v,A^{\complement} }(t)]=\mathbb{E}_{\la}[Y_{v, A}(t)]=\la P^{t}(v,A) \le \la |A| \sup_{x,y \in V}P^t(x,y) \le C \la |A|/\sqrt{t}, $$ where we have used the fact that for all $t \ge 1 $,   $\sup_{x,y} P^{t}(x,y) \le \sup_x P^{2\lfloor t/2 \rfloor }(x,x)  $ (see Proposition \ref{p:decayofreturn} below) and  $\sup_{x} P^t(x,x) \le C/ \sqrt{t} $ (\textit{e.g.},\ \cite{lyonsev}, where this is proved for SRW and lazy SRW with holding probability 1/2 -- the case of any other holding probability bounded away from 1 can be deduced from the SRW case, by averaging over the number of lazy steps  the walk performs by time $t$ and using the concentration of the Binomial distribution around its mean. Indeed, if $P$ is SRW and $P_p$ is lazy SRW with holding probability $p$ then $P_p^t(x,x)=\sum_i \Pr[\mathrm{Bin}(t,p)=i]P^i(x,x)$). \qed
\begin{proposition}
\label{p:decayofreturn}
SRW on a  regular graph satisfies $\sup_{x,y}P^t(x,y) \le \sup_{x}P^{2 \lfloor \sfrac{t}{2} \rfloor}(x,x)$.
\end{proposition}
\emph{Proof.} By reversibility (used in the second equality and to argue that $\sum_z [P^{t}(a,z)]^2=\sum_z P^{t}(a,z)P^{t}(z,a)=P^{2t}(a,a)$) and the Cauchy-Schwartz inequality (first inequality) 
\begin{equation}
\label{e:CSPxx}
\begin{split}
P^{2t}(x,y)&=\sum_z P^{t}(x,z)P^t(z,y)=\sum_z P^{t}(x,z)P^t(y,z) \\ & \le \sqrt{ P^{2t}(x,x)P^{2t}(y,y)}\le \sup_x P^{2t}(x,x). 
\\ \text{Similarly,} \quad  P^{2t+1}(x,y) & \le \sqrt{ P^{2t+2}(x,x)P^{2t}(y,y)} \le \sup_x P^{2t}(x,x). \quad \text{\qed}
\end{split}
\end{equation}
\emph{Proof of Lemma \ref{lem: infcollide}.}
The  fact that the  distribution of $N_t(\mathbf{w}) $ is Poisson follows from Poisson thinning. Let $M_t(\mathbf{w}):=\sum_{i=1}^t |\W_{\mathbf{w}_i}(i)| $.  By stationarity of the law of the occupation measure (Fact \ref{fact: thinning}) we have that $\mathbb{E}_{\la}[M_t(\mathbf{w})]=\la t$. Decomposing the last expectation according to the first time $i$ at which a walker is at $\mathbf{w}_i $ (and noting that  the contribution corresponding to  time $i$ is $\mathbb{(E}_{\la}[N_i(\mathbf{w})]-\mathbb{E}_{\la}[N_{i-1}(\mathbf{w})]) \sum_{j=0}^{t-i}P^j(\mathbf{w}_i,\mathbf{w}_{i+j})$, which can be bounded from above by $\mathbb{(E}_{\la}[N_i(\mathbf{w})]-\mathbb{E}_{\la}[N_{i-1}(\mathbf{w})])\sum_{j=0}^{t} \sup_{x,y \in V}P^j(x,y)$) we get that
\[\la t= \mathbb{E}_{\la}[M_t(\mathbf{w})] \le \mathbb{E}_{\la}[N_t(\mathbf{w})] \sum_{j=0}^{t} \sup_{x,y \in V}P^j(x,y) \le\mathbb{E}_{\la}[N_t(\mathbf{w})] C \sqrt{t}. \quad \text{\qed} \]    
\section{Appendix B: explicit construction of the SN model}
\label{A:construction}
\emph{Proof of Proposition \ref{prop: couplingforalllambda}.} For every $v \in V $ let $M_v(t)$ be a homogeneous Poisson process on $\R_{+}$ with rate $1 $ (all of which defined on the same probability space so that they are independent). For each $\la>0$, when the density of walkers is taken to be $\la$, we take $|\W_v^\la|:=M_v(\la)$, where $\W_v^\la$ denotes the the set of walkers whose initial position is $v$ (in the case of density $\la$).   Thus if $\la_1 < \la_2 $ then for all $v \in V $ we have that $\W_v^{\la_1} \subseteq \W_v^{\la_2}$. The assertion of the Proposition is already clear at this point. For the sake of completeness, we give additional details concerning the construction. 

We continue by constructing at each site $v$ an infinite collection of independent walks, where in practice, only $M_v(\la)$ of them shall be involved in the dynamics associated with the SN model with density $\la$.
For each $v \in V$ and $n \in \N $, let $\w_n^v=(\w_{n}^v(t))_{t \in \Z_+ }$ be a LSRW on $G$, started at $v$ (throughout we denote the law of such a walk by $\mathbb{P}_v$, where the holding probability is either clear from context or irrelevant). We take all the walks to be independent. Moreover, we take $\mathbf{W}:=(\w_n^v)_{n \in \N,v \in V} $ and $\mathbf{M}:=(M_v)_{v \in V} $ to be independent. We think of $\w_n^v $ as the walk performed by the $n$-th particle whose initial position is $v$.  

 We are now in the position to define $\ac_{t}^{\lambda}:=(V,E_{t,\lambda}) $.
Denote by  $Z_{i,j}^{u,v}(t) $ the indicator of the event that the $i$-th particle from $u$ met the $j$-th particle from $v$ by time $t$ (\textit{i.e.}, $Z_{i,j}^{u,v}(t):=\Ind{\w_{i}^u(s)=\w_{j}^v(s) \text{ for some }s \le t} $).  We want the last event to imply that $\{u,v
\} \in E_{t,\lambda}$ iff $i \le| \W_u^{\la}|=M_{u}(\lambda)  $ and $j \le |\W_v^{\la}|=M_{v}(\lambda)  $ (because we want the number of particles starting at each site which are involved in the dynamics to have a $\Pois(\la)$ distribution). Hence we define  $Q_{u,v}^{(\lambda)}(t)=\max\{Z_{i,j}^{u,v}(t)
:i \le | \W_u^{\la}|,j \le |\W_v^{\la}|\}$ (this is the indicator of the event that some $w \in \W_v^{\la}  $ met by time $t$ some $w' \in \W_u^{\la} $) and set $\{u,v
\} \in E_{t,\lambda}$ iff $Q_{u,v}^{(\lambda)}(t)=1$. \qed
\section{Appendix C: Proof of translation invariance and ergodicity}
\label{a:ergodicity}
Using the notation from \S\ref{sec: construction},
let $\mathbf{W}_{v}:=((\w_{i}^{v}(t))_{t \ge 0})_{i=1}^{|\W_{v}|}$
be the infinite walks that the walkers in $\W_v$ performed.
\begin{lemma}
\label{lem: approxbycylinder}
Let $G=(V,E)$ be an infinite connected  graph. Let  $\lambda
> 0$, $\epsilon >0$ and $t \in \Z_+ \cup \{\infty \} $.  Then for every $\mathcal{A} \in \F_{\mathrm{cylinder}}  $, there exist  a finite set $B=B(\mathcal{A},\epsilon,t) \subset V$ and $\mathcal{A}_{\epsilon}$  such that the event $\ac_{t}^{\lambda}(G)
\in \mathrm{graph}( \mathcal{A}_{\epsilon})  $ is in the $\sigma$-algebra generated by $(\mathbf{W}_{u}:u \in B)$ and \[ \Pr_{\lambda}[\ac_{t}^{\lambda}(G)
 \in \mathrm{graph}( \mathcal{A}_{\epsilon}
 \bigtriangleup \mathcal{A})  ] \le \epsilon.\]  
\end{lemma}
This follows via elementary measure theoretical considerations, and so we omit the proof.
\emph{Proof of Proposition \ref{prop: ergodicity}.}
We first establish translation invariance.  Let $\phi \in \mathrm{Aut}(G)$. We shall show that there exists a coupling of $\ac_{t}^{\lambda}(G) $ and $\phi(\ac_{t}^{\lambda}(G))$ (\textit{i.e.}, a probability space in which both are realized) such that $\ac_{t}^{\lambda}(G)=\phi(\ac_{t}^{\lambda}(G))$. This clearly implies the desired equality of the corresponding laws.

 Note that if $(\w(s))_{s \ge 0}$ has law $\mathbb{P}_v$ then $(\phi^{}(\w(s)))_{s\ge 0} $ has law $\mathbb{P}_{\phi^{}(v)}$. Recall the construction of the SN model from \S\ref{sec: construction} via  $(M_v,(
(\w_{n}^v(s))_{s \in \Z_+ })_{n \in \N})_{v \in V }$, where $(M_v)_{v \in V}$ are $\iid$ $\Pois(\la)$ and $(
(\w_{n}^v(s))_{s \in \Z_+ })_{n \in \N} $ are independent LSRWs started from $v$ (\textit{i.e.}, having law $\mathbb{P}_v$). Denote this realization of $\ac_t^{\la}(G) $ by $H:=(V,E(H))$. Now consider a different realization obtained by replacing for all $v \in V$ the walks $(
(\w_{n}^v(s))_{s \in \Z_+ })_{n \in \N}$ by $(
(\phi^{}(\w_{n}^{\phi^{-1}(v)}(s)))_{s \in \Z_+ })_{n \in \N}$ and replacing $M_v$ by $M_{\phi^{-1}(v)} $. Denote it  by $H':=(V,E(H'))$. Note that $\{u,v\} \in E(H')$ iff there is some $k \le M_{\phi^{-1}(u)} $,  $m \le M_{\phi^{-1}(v)} $ and $s \le t$ such that $\w_k^{\phi^{-1}(u)}(s)=\w_m^{\phi^{-1}(v)}(s) $. This occurs iff $\{\phi^{-1}(u),\phi^{-1}(v)\} \in E(H)$, or equivalently iff $\{u,v\} $ is an edge in $\phi(H)$. That is $H'=\phi(H)$.     

\medskip

We now prove ergodicity.
Let $\mathcal{A} \in \mathcal{I}$. Fix some $t \in
\Z_{+} \cup \{\infty
\}$. We seek to show that $\Pr_{\lambda}[\ac_{t}^{\la} \in \mathrm{graph} (\mathcal{A}) ] \in \{0,1\}$. Let $\epsilon>0$. By Lemma \ref{lem: approxbycylinder}, there exist a finite set $B_{\epsilon} \subset V$ and an event $\mathcal{A}_{\epsilon} $ such that $\{ \ac_{t}^{\lambda}(G)
\in {\mathrm{graph}}( \mathcal{A}_{\epsilon}) \} $  is in the
$\sigma$-algebra generated by $(\mathbf{W}_{b}(t))_{b \in B_{\epsilon}}$ and $\Pr_{\lambda}[ \ac_{t}^{\lambda}(G)
\in {\mathrm{graph}}( \mathcal{A}_{\epsilon}
 \bigtriangleup \mathcal{A})  ] \le \epsilon$. Let $r_{\epsilon}:= \max \{\mathrm{dist}(u,v):u,v \in B_{\epsilon} \}$. Let $\phi_{\epsilon} \in \mathrm{Aut}(G)$ be such that $\mathrm{dist}(v,\phi_{\epsilon}(v))>2r_{\epsilon}$ for all $v \in V$.

The event $\{ \ac_{t}^{\lambda}(G)
\in {\mathrm{graph}}( \phi_{\epsilon}( \mathcal{A}_{\epsilon}
 ))\}  $ is
in the
$\sigma$-algebra generated by $(\mathbf{W}_{\phi_{\epsilon}(b)}(t))_{b \in B_{\epsilon}}$.  By our choice of $\phi_{\epsilon}$, the sets $B_{\epsilon}$ and $\{\phi_{\epsilon}(b):b \in B_{\epsilon} \} $ are disjoint. Hence the events $\{ \ac_{t}^{\lambda}(G)\in {\mathrm{graph}}( \phi_{\epsilon}( \mathcal{A}_{\epsilon}
 ))\} $ and $\{ \ac_{t}^{\lambda}(G)\in {\mathrm{graph}}( \mathcal{A}_{\epsilon}
) \} $ are independent, as they depend on disjoint sets of walkers. By translation invariance and the fact that $\mathcal{A} \in \mathcal{I}$ (and so $\phi_{\epsilon}( \mathcal{A}_{\epsilon})
 \bigtriangleup \mathcal{A}=\phi_{\epsilon}( \mathcal{A}_{\epsilon}
 \bigtriangleup \mathcal{A})$)  \[\Pr_{\lambda}[ \ac_{t}^{\lambda}(G)
\in {\mathrm{graph}}( \phi_{\epsilon}( \mathcal{A}_{\epsilon})
 \bigtriangleup \mathcal{A})  ]=\Pr_{\lambda}[ \ac_{t}^{\lambda}(G)
\in {\mathrm{graph}}( \mathcal{A}_{\epsilon}
 \bigtriangleup \mathcal{A})  ] \le \epsilon. \] Hence $\Pr_{\lambda}[ \ac_{t}^{\lambda}(G) \in {\mathrm{graph(}}( \mathcal{\mathcal{A}_{\epsilon} \cap \phi_{\epsilon}(\mathcal{A}_{\epsilon}
 ))\bigtriangleup \mathcal{A}) }] \le 2 \epsilon $ and thus
\begin{equation*}
\begin{split}
&\Pr_{\lambda}[ \ac_{t}^{\lambda}(G) \in {\mathrm{graph}}( \mathcal{A})]=\lim_{\epsilon \to 0}\Pr_{\lambda}[ \ac_{t}^{\lambda}(G) \in {\mathrm{graph}}( \mathcal{\mathcal{A}_{\epsilon} \cap \phi_{\epsilon}(\mathcal{A}_{\epsilon}
 )) }] \\ &= \lim_{\epsilon \to 0}\Pr_{\lambda}[ \ac_{t}^{\lambda}(G)\in {\mathrm{graph}}(\mathcal{A}_{\epsilon})]\Pr_{\lambda}[ \ac_{t}^{\lambda}(G)\in {\mathrm{graph}}(\mathcal{\phi_{\epsilon}(A}_{\epsilon}))] \\& =\lim_{\epsilon \to 0}\Pr_{\lambda}[ \ac_{t}^{\lambda}(G)
\in {\mathrm{graph}}( \mathcal{A}_{\epsilon})]^{2} =\Pr_{\lambda}[ \ac_{t}^{\lambda}(G) \in {\mathrm{graph}}( \mathcal{A})]^{2}.
\end{split}
\end{equation*}
Thus indeed $\Pr_{\lambda}[ \ac_{t}^{\lambda}(G) \in  {\mathrm{graph}}(  \mathcal{A})] \in \{0,1\}$, as desired. \qed

\section{Proof of Proposition \ref{prop: dominationbybrw}}
\label{a:p8.2}
We denote the walk performed by a walker $w$ by $(\mathbf{w}(t))_{t \ge 0} $. Recall that  $\W_v(t)$ is the set of walkers whose location at time $t$ is $v $ and that for $B \subseteq V $ and $t \ge 0$,  $\W_{B}(t):=\cup_{b \in B}\W_b(t)$ is the set of walkers occupying $B$ at time $t$. We denote the lexicographic order by $\prec $. Our use of the lexicographic order below is just a mean of preforming the bookkeeping in a manner which avoids double-counting (so that each walker is recruited to the exploration process at most once). It plays no additional role in the argument.

\medskip

{\em Proof of Proposition 5.2:}
At stage zero, we start the exploration process of $\fc(o)$ by setting $\W_{0,0}:=\W_o
$ and $\AAA_{0,0}=\{v \} $. We label the walkers in $\W_{0,0}$ as $w_{0,0,1},\ldots,w_{0,0,|\W_{0,0}|}$. 

 If $\W_o$ is empty the exploration process is completed. Otherwise, at stage one we set \[\AAA_{1,1}:=\{\mathbf{w}(1):w \in \W_{0,0} \}  \quad \text{and} \quad 
\W_{1,1}:=\{w \in \W_{\AAA_{1,1}}(1):w \notin \W_{0,0}  \}\] to be the collection of walkers not belonging to $\W_{0,0}
$, which have the same position at time 1 as some walker in $\W_{0,0}$. We say that $w \in \W_{1,1}$ is an offspring of $w_{0,0,j}$ if $\mathbf{w}(1)=\w_{0,0,j}(1) $ and $j$ is the minimal integer such that this holds. Finally, we label the elements of $\W_{1,1} $ as $w_{1,1,1},\ldots,w_{1,1,|\W_{1,1}|}$.

The first ``interesting" stage of the process is stage 2, thus we describe it before proceeding to the description of a general stage. Let \[\AAA_{2,0}:=\{\mathbf{w}(0):w \in \W_{1,1} \} \quad \text{and} \quad  \AAA_{2,2}:=\{\mathbf{w}(2):w \in \W_{0,0} \cup \W_{1,1} \} .\] We set \[\W_{2,0}:=\{w \in \W_{\AAA_{2,0}}(0):w \notin \W_{1,1}    \} \quad \text{and} \quad \W_{2,2}:=\{w \in \W_{\AAA_{2,2}}(2):w \notin  \cup_{(i,j)\prec (2,2)} \W_{i,j}  \}.\]
In words, $\W_{i,j}$ is the set of walkers recruited to the process at stage $i$ of the exploration process, by meeting at time $j$ some walker which was  recruited to the exploration process at an earlier stage (not necessarily an earlier time). These are the walkers which at time $j $ visit the set $\AAA_{i,j}$ but for all $(i',j') \prec (i,j) $ avoided $\AAA_{i',j'}$ at time $j'$. Once a walker is recruited to the exploration process by joining $\W_{i,j}$ at stage $i$, we then expose at each stage $i+ \ell $ (where $\ell \in \N$) its location at time $j+\ell $ and if $\ell \le j$ we also expose its position at time $j-\ell$. 

In particular, for every $t$, for some values of $s  $ (namely, for $s \le t $ such that $t-s$ is even) we expose at the $t$-th stage of the exploration process
the location at time $s$ of some particles which have been recruited to the exploration process prior to stage $t$ (namely of the ones in $\W_{i,j}$ for $(i,j)$ such that either $j+(t-i)=s$ or $j-(t-i)=s$).  We denote the collection of these locations by \[\AAA_{t,s}:=\cup_{(i,j): \, j+(t-i)=s \text{ or }j-(t-i)=s } \{\w(s):w \in \W_{i,j} \}.\] Finally, we let $\W_{t,s}$ be the collection of walkers in $\W_{\AAA_{t,s}}(s)$ (\textit{i.e.}, the ones occupying $\AAA_{t,s}$ at time $s$) which do not belong to $\W_{t',s'}$ for any $(t',s') \prec (t,s) $.

 The parent of $w \in \W_{2,0}$ (respectively, $\W_{2,2}$) is defined to be $w_{1,1,k} \in \W_{1,1} $ (respectively, $w_{i,j,k} \in \W_{0,0} \cup \W_{1,1} $)  such that $\w(0)=\w_{1,1,k}(0) $ (respectively, $\w(2)=\w_{i,j,k}(2) $) and $(1,1,k)$ (respectively, $(i,j,k)$) is minimal w.r.t.~$\prec$.  Finally, for $(i,j)\in \{(2,0),(2,2)\}$ we label the walkers in $\W_{i,j}$ as $w_{i,j,1},\ldots,w_{i,j,|\W_{i,j}|}$.

The sets $\AAA_{r,s}$ and $\W_{r,s}=\{ w_{r,s,1},\ldots,w_{r,s,|\W_{r,s}|} \} $ (where $0 \le s \le r $ is of the same parity as $r$) are defined inductively so that the following holds:
\begin{itemize}
\item[(1)] $\AAA_{r,s}:=\{\w(s):w \in \cup_{(i,j)\in F_{r,s} \cup B_{r,s}} \W_{i,j} \} $, where \[F_{r,s}:=\{(i,j):(i,j) \prec (r,s) \text{ and } r-i=s-j>0  \} \quad \text{and} \] \[ B_{r,s}:=\{(i,j): j-s=r-i>0 \}.\] 
(in simple words, as described above, $\AAA_{r,s}$ are the positions explored by the exploration process at stage $r$ corresponding to time $s$ of some walkers. This walkers were  recruited at an earlier stage, either at an earlier time or at a latter time.  If they were recruited at stage $i$ and time $j$ then by construction in the first case $(i,j) \in F_{r,s} $, while in the second case  $(i,j) \in B_{r,s} $.)
\item[(2)] $\W_{r,s}:=\{w \in \W_{\AAA_{r,s}}(s): w \notin  \cup_{(i,j)\prec (r,s)} \W_{i,j}   \}$. Note that this is the set of walkers which joined the exploration process at stage $r$ and time $s$.
\end{itemize}
It follows that \[\cup_{(r,s):\, r \ge s,\, r-s \text{ is even}}\W_{r,s}=\fc(o) .\]

We now describe the assignment of offspring to walkers.
In the $r$-th stage we expose the sets  $\AAA_{r,s} $ (where $0 \le s \le r $ is of the same parity as $r$) sequentially according to the order $\prec$.  
We expose each $\AAA_{r,s} $ by exposing sequentially the positions of the walkers in $\cup_{(i,j)\in F_{r,s} \cup B_{r,s}} \W_{i,j}$ one walker at a time, according to the order $\prec$ (over the indices of the walkers $(i,j,k)$ such that $(i,j)\in  F_{r,s} \cup B_{r,s}  $ and $1 \le k \le |\W_{i,j}| $). We say that $w \in \W_{r,s} $ is an offspring of $w_{i,j,k}$ (where $(i,j)\in  F_{r,s} \cup B_{r,s}  $ and $1 \le k \le |\W_{i,j}| $) if $\w(s)=\w_{i,j,k}(s) $ but for all  $(i',j',k') \prec (i,j,k) $ such that $(i',j')\in  F_{r,s} \cup B_{r,s}  $ and $1 \le k' \le |\W_{i',j'}| $ we have that $\w(s) \neq \w_{i',j',k'}(s) $. Moreover, as $w \notin  \cup_{(i',j')\prec (r,s)} \W_{i',j'}$ (by the definition of $\W_{r,s}$ and the assumption that $w \in \W_{r,s}$), we also have that $\w(\ell) \notin \AAA_{n,\ell}$ for all $0 \le \ell \le n \le r $ (where $n-\ell$ is even) such that $(n,\ell) \prec (r,s) $. If $s > j$ (respectively, $j>s $) we say that $w $ is a \emph{forward} (respectively, \emph{backward}) offspring of $w_{i,j,k}$. Let $\BB_{i,j,k}(r-i)$ and $\F_{i,j,k}(r-i)$ be the backward and forward (resp.) offspring of $w_{i,j,k}$ at stage $r$. Denote by $B_{i,j,k}(r-i) $ and $F_{i,j,k}(r-i)$ the collection of space-time coordinates which (as described above) a walker in  $\BB_{i,j,k}(r-i)$ and $\F_{i,j,k}(r-i)$ (respectively) has to avoid, in order to have not been recruited to the exploration process prior to the exposure of   $\BB_{i,j,k}(r-i)$ or $\F_{i,j,k}(r-i)$, respectively  (namely, these are the space-time coordinates exposed prior to the exposure of $\BB_{i,j,k}(r-i) $ and $\F_{i,j,k}(r-i)$, respectively).  

We think of a walker $w_{i,j,k}$ as performing a forward walk, $\mathbf{fw}_{i,j,k}(\ell):=\w(j+\ell) $ and a backward walk (of length $j$) $\mathbf{bw}_{i,j,k}(\ell)=\w(j-\ell) $. At each stage $r \ge i$ we expose one additional step of $\mathbf{fw}_{i,j,k} $ (namely, $\mathbf{fw}_{i,j,k}(r-i)=\w(j+(r-i))$) and if $j \ge r-i$ also one additional step of $\mathbf{bw}_{i,j,k} $ (namely, $\mathbf{bw}_{i,j,k}(r-i)=\w(j-(r-i))$). Note that the forward (respectively, backward) offspring of $w_{i,j,k}$ at stage $r$ are precisely the collection of all walkers $w$ whose location at time $j+r-i$ (respectively, $j-r+i$) is $\w_{i,j,k}(j+r-i)$ (respectively, $\w_{i,j,k}(j-r+i) $) so that $(\w(\ell),\ell) \notin F_{i,j,k}(r-i)  $  (respectively, $\notin B_{i,j,k}(r-i) $) for all $0 \le \ell \le r$.

Recall that $\Gamma_{r}$ is the collection of all walks of length $r$ in $G$ and that for $\gamma \in \Gamma_r$, we denote the number of walkers which performed the walk $\gamma$ by $X_{\gamma} \sim \Pois(\la p(\gamma))$, where $p(\gamma):=\prod_{i=0}^{r-1}P(\gamma_i,\gamma_{i+1})$ .  

 Let $\Gamma_{i,j,k,r,\mathrm{f}}$ (respectively, $\Gamma_{i,j,k,r,\mathrm{b}}$) be the collection of all  $\gamma=(\gamma_0,\ldots,\gamma_r) \in \Gamma_r$ such that~$(\gamma_\ell,\ell) \notin F_{i,j,k}(r-i) $ for all $\ell$ and $\gamma_{j+r-i}=\w_{i,j,k}(j+r-i) $ (respectively, $(\gamma_\ell,\ell) \notin B_{i,j,k}(r-i) $ for all $\ell$ and $\gamma_{j-r+i}=\w_{i,j,k}(j-r+i) $).  By Poisson thinning, given $F_{i,j,k}(r-i)$ and $\w_{i,j,k}(j+r-i) $ (respectively, $B_{i,j,k}(r-i)$ and $\w_{i,j,k}(j-r+i)$), $(X_{\gamma})_{\gamma \in \Gamma_{i,j,k,r,\mathrm{f}}}$ (respectively, $(X_{\gamma})_{\gamma \in \Gamma_{i,j,k,r,\mathrm{b}}} $) are independent Poisson r.v.'s with mean $\la p(\gamma) $, respectively.

Now, consider the case that after exposing $\w_{i,j,k}(j+r-i)$ (respectively, $\w_{i,j,k}(j-r+i)$), for each \[\gamma \in \{ \gamma' \in \Gamma_r : \gamma_{j+r-i}'=\w_{i,j,k}(j+r-i) \} \setminus \Gamma_{i,j,k,r,\mathrm{f}}  \] \[\text{(respectively,} \quad \gamma \in\{ \gamma' \in \Gamma_r : \gamma_{j-r+i}'=\w_{i,j,k}(j-r+i) \} \setminus \Gamma_{i,j,k,r,\mathrm{b}}   \text{)} \] we ``plant" $\la p(\gamma) $ new ``dummy particles" (independently for different such $\gamma$'s) which perform the path $\gamma$, and then continue their walk after time $r$ randomly. The dummy particles do not discover new walkers in the following stages of the exploration process (\textit{i.e.},they do not have any offspring and the trajectory of their walk plays no role in the following stages). If we count the dummy particles as part of the  offspring of $w_{i,j,k}$ corresponding to its forward step at stage $r$, then we have that $(X_{\gamma}^{i,j,k,r,\mathrm{f}} )_{\gamma \in \Gamma_r : \gamma_{j-r+i}=\w_{i,j,k}(j-r+i)}$ are independent Poisson r.v.'s and $\mathbb{E}_{\la}[X_{\gamma}^{i,j,k,r,\mathrm{f}} ]=\la p(\gamma) $ for all $\gamma \in\{ \gamma' \in \Gamma_r : \gamma_{j-r+i}'=\w_{i,j,k}(j-r+i) \}$, where $X_{\gamma}^{i,j,k,r,\mathrm{f}}$ is the  number of offspring of $w_{i,j,k}$ corresponding to its forward step at stage $r$ who perform the walk $\gamma$. By Poisson thinning, this is the same as having $\Pois(\la)$ offspring, each performing an independent $\Z_+$-indexed LSRW on $G$, conditioned to be at $\w_{i,j,k}(j+r-i)$ at time $j+r-i$. A similar statement holds for the  number of offspring of $w_{i,j,k}$ corresponding to its backwards step at stage $r$.

Recall the construction of the sets $\mathcal{V}_{r,s}=\{v_{r,s,1},\ldots,v_{r,s,|\mathcal{V}_{r,s}|} \}$ from the equivalent representation of $\brw$. It is not hard to prove that the sets $\W_{r,s}$ and $\mathcal{V}_{r,s} $ can be coupled (for all $0 \le s \le r$ so that $r-s$ is even) so that $\W_{r,s} \subseteq \mathcal{V}_{r,s} $. More  precisely, this can be done so that for all $k \le |\W_{r,s}|$ we have that  $\mathbf{v}_{r,s,k}(t)=\w_{r,s,k}(t)$ for all $t \ge 0$. We leave the details to the reader.  \qed

\end{document}